\documentclass{amsart}

\usepackage{amsmath}
\usepackage{amsthm}
\usepackage{hyperref}
\usepackage{amsfonts,graphics,amsthm,amsfonts,amscd,latexsym}
\usepackage{epsfig}
\usepackage{flafter}
\usepackage{mathtools}
\usepackage{comment}
\usepackage{stmaryrd}

\usepackage{mathabx,epsfig}

\hypersetup{
    colorlinks=true,    
    linkcolor=blue,          
    citecolor=blue,      
    filecolor=blue,      
    urlcolor=blue           
}
\usepackage{tikz}
\usetikzlibrary{graphs,positioning,arrows,shapes.misc,decorations.pathmorphing}

\tikzset{
    >=stealth,
    every picture/.style={thick},
    graphs/every graph/.style={empty nodes},
}

\tikzstyle{vertex}=[
    draw,
    circle,
    fill=black,
    inner sep=1pt,
    minimum width=5pt,
]
\usepackage[position=top]{subfig}
\usepackage{amssymb}
\usepackage{color}

\calclayout

\usetikzlibrary{decorations.pathmorphing}
\tikzstyle{printersafe}=[decoration={snake,amplitude=0pt}]

\newcommand{\supp}{\operatorname{supp}}

\newcommand{\pp}{\mathbb{P}}

\newcommand{\qq}{\mathbb{Q}}

\newcommand{\nn}{\mathbb{N}}
\newcommand{\rr}{\mathbb{R}}

\newcommand{\kk}{\mathbb{K}}

\def\O#1.{\mathcal {O}_{#1}}			
\def\pr #1.{\mathbb P^{#1}}				
\def\af #1.{\mathbb A^{#1}}			
\def\ses#1.#2.#3.{0\to #1\to #2\to #3 \to 0}	
\def\xrar#1.{\xrightarrow{#1}}			
\def\K#1.{K_{#1}}						
\def\bA#1.{\mathbf{A}_{#1}}			
\def\bM#1.{\mathbf{M}_{#1}}				
\def\bL#1.{\mathbf{L}_{#1}}				
\def\bB#1.{\mathbf{B}_{#1}}				
\def\bK#1.{\mathbf{K}_{#1}}			
\def\subs#1.{_{#1}}					
\def\sups#1.{^{#1}}

\usepackage{tikz}
\usetikzlibrary{matrix,arrows,decorations.pathmorphing}

  \newtheorem{theorem}{Theorem}[section]

  \newtheorem{lemma}[theorem]{Lemma}
  \newtheorem{proposition}[theorem]{Proposition}

  \newtheorem{conjecture}[theorem]{Conjecture}

  \newtheorem{definition}[theorem]{Definition}
  \newtheorem{example}[theorem]{Example}

  \newtheorem{question}[theorem]{Question}

\theoremstyle{remark}

\numberwithin{equation}{section}

\usepackage[all]{xy}

\begin{document}

\title[Log Calabi--Yau pairs of birational complexity zero]{Log Calabi--Yau pairs of birational complexity zero}

\author[J. Enwright]{Joshua Enwright}
\address{UCLA Mathematics Department, Box 951555, Los Angeles, CA 90095-1555, USA
}
\email{jlenwright1@math.ucla.edu}

\author[F.~Figueroa]{Fernando Figueroa}
\address{Department of Mathematics, Princeton University, Fine Hall, Washington Road, Princeton, NJ 08544-1000, USA
}
\email{fzamora@princeton.edu}

\author[J.~Moraga]{Joaqu\'in Moraga}
\address{UCLA Mathematics Department, Box 951555, Los Angeles, CA 90095-1555, USA
}
\email{jmoraga@math.ucla.edu}


\begin{abstract}
In this article, we study the geometry of log Calabi--Yau pairs
$(X,B)$ of index one and birational complexity zero. 
Firstly, we propose a conjecture that characterizes such pairs $(X,B)$
in terms of their dual complex
and the rationality of their log canonical places.
Secondly, we show that for these pairs the open set $X\setminus B$ is divisorially covered by open affine subvarieties which are isomorphic to open subvarieties of algebraic tori.
We introduce and study invariants that measure the geometry and the number of these open subvarieties of algebraic tori.
Thirdly, we study boundedness properties of log Calabi--Yau pairs of index one and birational complexity zero. For instance, 
in dimension $2$ we prove that such pairs are affinely bounded.
\end{abstract}

\maketitle

\setcounter{tocdepth}{1} 
\tableofcontents

\section{Introduction}

Calabi--Yau varieties are one of the three building blocks of smooth projective varieties. 
Log Calabi--Yau pairs $(X,B)$ are geometric objects that appear as degenerations of Calabi--Yau varieties.
The {\em complexity} of a log Calabi--Yau pair $(X,B)$ is defined to be
\[
c(X,B):= \dim X + \dim {\rm Cl}_\qq(X) -|B|,
\]
where $|B|$ is the sum of the coefficients of $B$.
The complexity of a log Calabi--Yau pair is non-negative and it determines whether the log Calabi--Yau pair is toric (see ~\cite[Theorem 1.2]{BMSZ18}).
The {\em birational complexity}, denoted by $c_{\rm bir}(X,B)$, of a log Calabi--Yau pair $(X,B)$, is the infimum among the complexities of birational models of $(X,B)$.
The birational complexity has recently been introduced by Mauri and the third author 
as a proxy to study the dual complex of log Calabi--Yau pairs~\cite{MM24}.
Log Calabi--Yau pairs of index one\footnote{The {\em index} of a log Calabi--Yau pair $(X,\Delta)$ is the smallest positive integer $m$ for which $m(K_X+B)\sim 0$.} and birational complexity zero are precisely those that admit birational toric models~\cite[Theorem 1.6]{MM24}.
In this article, we study the geometry of log Calabi--Yau pairs of index one and birational complexity zero.

\subsection{Birational complexity zero}
In this subsection, we discuss a fundamental conjecture and a structural theorem for log Calabi--Yau pairs of birational complexity zero.

In what follows, we write $\Sigma^n$ for the sum of the hyperplane coordinates of the $n$-dimensional projective space $\mathbb{P}^n$. 
Thus, the log pair $(\pp^n,\Sigma^n)$ is a log Calabi--Yau pair of index one.
We propose the following conjecture that characterizes log Calabi--Yau pairs
admitting toric models, i.e., 
log Calabi--Yau pairs of birational complexity zero and index one.

\begin{conjecture}\label{conj:log-rationality}
Let $(X,B)$ be a log Calabi--Yau pair of dimension $n$.
Assume that the following conditions are satisfied:
\begin{enumerate}
\item we have $\mathcal{D}(X,B)\simeq_{\rm PL} S^{n-1}$; and 
\item each dlt stratum of $(X,B)$ is rational.
\end{enumerate} 
Then, we have $c_{\rm bir}(X,B)=0$.
In particular, we have a crepant
birational map 
$(\pp^n,\Sigma^n)\dashrightarrow (X,B)$.
\end{conjecture} 

In the previous conjecture, the dlt strata of $(X,B)$ are the strata of a dlt modification of the pair (see Definition~\ref{def:dlt-stratum}).
Let us emphasize that due to ~\cite[Corollary 4.(i)]{FMM22} and ~\cite[Proposition 3]{KX16}, Conjecture~\ref{conj:log-rationality}.(1) implies that $K_X+B\sim 0$. This explains the last line in the statement of the conjecture.
Conjecture~\ref{conj:log-rationality} is known in dimension $2$ by the work of Gross, Hacking, and Keel~\cite[Proposition 1.3]{GHK15}.
Furthermore, Conjecture~\ref{conj:log-rationality} is known for $X\simeq \pp^3$ by the work of Ducat~\cite[Theorem 1.2]{Duc22}.
It is worth mentioning that the previous conjecture fails if we drop any of its hypotheses.
In~\cite[\S 3]{Kal20}, Kaloghiros gives an example of a log Calabi--Yau pair $(X,B)$ with $\mathcal{D}(X,B)\simeq_{\rm PL} S^2$ and 
$X$ is birationally superrigid.
In particular, the pair $(X,B)$ does not have a toric model, so its birational complexity is positive.
In~\cite[Theorem 1.18]{MM24}, the authors give examples, in each dimension $n\geq 3$, of log Calabi--Yau pairs $(X_n,B_n)$
of dimension $n$, birational complexity $n$,
rational dlt strata, and 
dual complex $\mathcal{D}(X_n,B_n)\simeq_{\rm PL} \mathbb{P}^{n-1}_\rr$.

We turn to state a structural theorem for log Calabi--Yau pairs of
index one and birational complexity zero.
Toric varieties are covered by algebraic tori with monomial gluing functions. Cluster varieties can be regarded as generalizations of toric varieties. 
A cluster variety admits a big open subset
that is covered by algebraic tori. 
In the case of cluster varieties, the gluing functions are not necessarily monomial, however, they preserve the volume function of the tori.
The following theorem 
is a generalization of the previous facts 
to the case of log Calabi--Yau pairs
of birational complexity zero.
As it is expected, in this broader setting, one needs to consider more general affine varieties to perform these coverings.

\begin{theorem}\label{introthm:covering-bcomp0}
Let $(X,B)$ be a log Calabi--Yau pair of 
index one and birational complexity zero.
Then, there is a finite set
$\{U_i\}_{i\in I}$ of open affine subvarieties of $X\setminus B$ 
satisfying the two following conditions:
\begin{enumerate}
\item the set $\{U_i\}_{i\in I}$ divisorially covers $X\setminus B$;\footnote{We say that an algebraic variety $V$ is {\em divisorially covered} by the open sets $U_1,\dots,U_k$ if ${\rm codim}_{V}\left(V\setminus \bigcup_{i=1}^k U_i\right)\geq 2$. 
} and
\item each $U_i$ is isomorphic to an open subset of $\mathbb{G}_m^n$.
\end{enumerate} 
\end{theorem}

The open sets in Theorem~\ref{introthm:covering-bcomp0} appear as $\phi(\mathbb{G}_m^n\setminus {\rm Ex}(\phi))$ where
$\phi\colon (\pp^n,\Sigma^n)\dashrightarrow (X,B)$ is a crepant birational map.
We say that a finite set $\{\phi_i\}_{i \in I}$ of crepant birational maps
$\phi_i\colon (\pp^n,\Sigma^n)\dashrightarrow (X,B)$
{\em divisorially covers} $(X,B)$
if the open affines $U_i:=\phi_i(\mathbb{G}_m^n\setminus {\rm Ex}(\phi_i))$ divisorially cover $X\setminus B$ (see Definition~\ref{def:div-cover-logCY}).

Theorem~\ref{introthm:covering-bcomp0} states that
whenever $c_{\rm bir}(X,B)=0$ the variety
$X\setminus B$ is divisorially covered by
open subsets $U\subseteq \mathbb{G}_m^n$. 
The geometry of $U$ is determined by its complement on $\mathbb{G}_m^n$. We turn to study this complement in the following subsection.

\subsection{The torus exceptional degree}
In this subsection, we study the affine open subvarieties covering log Calabi--Yau pairs of birational complexity zero.
Let $(X,B)$ be a log Calabi--Yau pair
and let $\varphi\colon (\pp^n,\Sigma^n) \dashrightarrow 
(X,B)$ be a crepant birational map. 
We define the {\em torus exceptional degree}
of $\varphi$ to be
\[
\mathcal{T}(\varphi):={\rm deg}_{\pp^n}\left(
\overline{{\rm Ex}^1(\varphi)|_{\mathbb{G}^n_m}} 
\right) 
\]
where ${\rm Ex}^1(\varphi)$ is the reduced divisorial part
of the exceptional locus of $\varphi$.
Note that the torus exceptional degree is zero if and only if $\phi$ induces an embedding of a big open subset of the algebraic torus $\mathbb{G}_m^n$.
A big open subset of an algebraic torus will be called an {\em almost torus}.
The {\em torus exceptional degree} of $(X,B)$
is defined to be:
\[
\mathcal{T}(X,B):=\min\{ 
\mathcal{T}(\varphi) \mid 
\text{
$\varphi\colon (\pp^n,\Sigma^n)\dashrightarrow (X,B)$
is a crepant birational map
}
\}. 
\]
In the case that the set in the definition is empty, we will set $\mathcal{T}(X,B)=\infty$.
If $(X,B)$ is a log Calabi--Yau pair
of index one, 
then the torus exceptional degree $\mathcal{T}(X,B)$ is finite
if and only if $c_{\rm bir}(X,B)=0$. 
The {\em total torus exceptional degree}
is defined to be:
\[
\mathcal{T}_t(X,B)={\rm min}\left\{ \sum _{i \in I} \mathcal{T}(\varphi _i) \,\middle|\, \{ \varphi\}_{i \in I} \text{ is a set of crepant birational maps divisorially covering $(X,B)$}
\right\}.
\] 
In Subsection~\ref{subsec:ted}, we expand on the aforementioned definitions.
The following gives a structural theorem for the torus exceptional degree
and the total torus exceptional degree.

\begin{theorem}\label{introthm:t-equal-zero}
Let $(X,B)$ be a log Calabi--Yau pair of index one.
The following statements hold:
\begin{enumerate} 
\item we have $\mathcal{T}_t(X,B)\geq\mathcal{T}(X,B)\geq 0$;
\item the value $\mathcal{T}_t(X,B)$ is finite if and only if $(X,B)$ has birational complexity zero;
\item $\mathcal{T}(X,B)=0$ if and only if 
$(X,B)$ is divisorially covered by finitely many almost tori; and 
\item If $X$ is a $\qq$-factorial Fano variety and $\mathcal{T}(X,B)=0$, then 
the log Calabi--Yau pair $(X,B)$ is covered by finitely many algebraic tori.
\end{enumerate} 
\end{theorem} 

In dimension two, we prove that the torus exceptional degree is either infinite or bounded above by $9$.

\begin{theorem}\label{introthm:bounded-t-surf}
Let $(X,B)$ be a log Calabi--Yau surface of index one and birational complexity zero. Then, we have 
$\mathcal{T}(X,B)\leq 9$.
\end{theorem}

The previous theorem states that log Calabi--Yau surfaces of birational complexity zero
admit affine charts of the form $\phi(\mathbb{G}_m^2\setminus {\rm Ex}(\phi))$ where $\phi\colon \mathbb{G}_m^2 \dashrightarrow X\setminus B$ is a birational map and the degree of ${\rm Ex}(\phi)$ is controlled above.
In Example~\ref{ex:total-ted-unbounded},
we show that there is a sequence log Calabi--Yau pairs $(X_n,B_n)$ of dimension $2$, index one, and birational complexity zero, for which the sequence $\mathcal{T}_t(X_n,B_n)$ diverges to infinity. 
Theorem~\ref{introthm:bounded-t-surf} will be a consequence of the following explicit result for Gorenstein del Pezzo surfaces of rank one.

\begin{theorem}\label{introthm:bound-ted-gor-dP}
Let $X$ be a Gorenstein del Pezzo surface of Picard rank one.
Let $(X,B)$ be a log Calabi--Yau pair 
of index one and birational complexity zero.
Then, we have $\mathcal{T}(X,B)\in \{0,1,\dots,9\}$.
\end{theorem}

We propose the following conjecture 
which predicts that the torus exceptional degree is bounded above in terms of the dimension.

\begin{conjecture}\label{conj:upper-bound-ted}
Let $n$ be a positive integer.
There exists a positive integer $t_n$, only depending on $n$, satisfying the following.
Let $(X,B)$ be a log Calabi--Yau pair of dimension $n$, index one, and birational complexity zero. 
Then, we have $\mathcal{T}(X,B)\leq t_n$.
\end{conjecture}

\subsection{Connections with boundedness}
In this subsection, we draw a connection between log Calabi--Yau pairs of birational complexity zero and boundedness of algebraic varieties. 
First, we introduce the concept of {\em affinely bounded} families of projective varieties.

\begin{definition}
{\em 
Let $\mathcal{F}$ be a family of projective varieties. 
We say that $\mathcal{F}$ is {\em affinely bounded} if there exists a scheme $T$ of finite type over the ground field $\kk$ and a finite type affine morphism $\mathcal{U}\rightarrow T$
such that every element $X\in \mathcal{F}$
admits an open affine set isomorphic to $\mathcal{U}_t$ for some $t\in T$.
}
\end{definition}

Bounded families of projective varieties are affinely bounded.
In Example~\ref{ex:affine-bounded-not-bounded}, we recall that the converse is not true.
Further, affinely bounded families are birationally bounded.
In Example~\ref{ex:bir-bound-not-aff-bound}, we show that there are families of rational varieties that are not affinely bounded.
As an application of Theorem~\ref{introthm:bounded-t-surf}, we obtain the following theorem.

\begin{theorem} 
\label{introthm:bcomp-zero-surf-aff-bounded}
Let $\mathcal{F}_2$ be the family of projective surfaces $X$ satisfying the following conditions:
\begin{itemize}
    \item there is a log Calabi--Yau pair $(X,B)$ with $K_X+B\sim 0$; and 
    \item we have $c_{\rm bir}(X,B)=0$.
\end{itemize}
Then the family $\mathcal{F}_2$ is affinely bounded.
\end{theorem}

We expect the previous theorem to generalize to higher dimensions.

\begin{conjecture} 
\label{introthm:bcomp-zero-aff-bounded}
Let $\mathcal{F}_n$ be the family of $n$-dimensional projective varieties $X$ satisfying the following conditions:
\begin{itemize}
    \item there exists a log Calabi--Yau pair $(X,B)$ with $K_X+B\sim 0$; and 
    \item we have $c_{\rm bir}(X,B)=0$.
\end{itemize}
Then the family $\mathcal{F}_n$ is affinely bounded.
\end{conjecture}

\subsection{The cluster type}
In this subsection, we introduce the concept of cluster type log Calabi--Yau pairs and prove some structural theorems about these pairs.

We say that a log Calabi--Yau pair $(X,B)$ is {\em of cluster type} if $\mathcal{T}(X,B)=0$.
By Theorem~\ref{introthm:t-equal-zero}, 
for a log Calabi--Yau pair $(X,B)$ of cluster type the open subvariety $X\setminus B$ is covered by finitely many almost tori.
In this case, the {\em cluster type} of $(X,B)$,
denoted by $\mathcal{C}(X,B)$, is the smallest number of almost tori
needed to divisorially cover $X\setminus B$.
If a log Calabi--Yau pair is not of cluster type, then we set $\mathcal{C}(X,B)=\infty$.
The following theorem characterizes 
log Calabi--Yau toric pairs among
pairs of cluster type.

\begin{theorem}\label{theorem:cluster-type-1}
Let $(X,B)$ be a log Calabi--Yau pair. 
Then, we have that 
$\mathcal{C}(X,B)\geq 1$.
Furthermore, the equality holds
if and only if the pair $(X,B)$ is a
toric log Calabi--Yau pair.
\end{theorem}

In Example~\ref{ex:large-cluster-type}, 
we show that for every $c\geq 1$
there exists a log Calabi--Yau surface 
$(X,B)$ of cluster type
with $\mathcal{C}(X,B)=c$.
The following theorem states that such a phenomenon does not happen among Fano surfaces.

\begin{theorem}\label{introthm:cluster-Fsurf}
Let $X$ be a Fano surface and
let $(X,B)$ be a log Calabi--Yau pair.
Then we have 
$\mathcal{C}(X,B)\in \{1,2,\infty\}$.
\end{theorem}

The previous theorem states that Fano surfaces of cluster type are quite similar to Fano toric surfaces.
In particular, for every Fano surface $X$ of cluster type, we can write
\[
X=\mathbb{G}_m^2 \cup_\phi \mathbb{G}_m^2 \sqcup_{i=1}^j \mathbb{G}_{m,i} 
\sqcup_{i=1}^k \{x_i\}
\]
where $\phi^*\Omega=\Omega$ with $\Omega=\frac{dt_1\wedge dt_2}{t_1t_2}$
and the $x_1,\dots,x_k\in X$ are closed points.

\subsection*{Acknowledgements}

The authors would like to thank 
Konstantin Loginov, Mirko Mauri, and Burt Totaro, for very useful comments. The first author is partially supported by NSF grant DMS-2054553.

\section{Preliminaries}
In this section, we collect some preliminaries about log Calabi--Yau pairs, birational transformations, and bounded families. 
We work over an algebraically closed field $\kk$ of characteristic zero. 
The least number of generators of a group $G$, will be denoted as $d(G)$.
We denote by $\Sigma^n$ the sum of the coordinate hyperplanes of $\pp^n$. 

\subsection{Log Calabi--Yau pairs}
In this subsection, we recall some basic concepts related to log Calabi--Yau pairs.

\begin{definition}
{\em
A {\em log pair} is a couple $(X,B)$ where $X$ is a normal quasi-projective variety and $B$ is an effective $\qq$-divisor for which $K_X+B$ is $\qq$-Cartier.
Let $\pi\colon Y\rightarrow X$ be a projective birational morphism from a normal variety and
$E\subset Y$ be a prime divisor.
The {\em log discrepancy} of $(X,B)$ at $E$, denoted by $a_E(X,B)$ is the rational number
\[
1-{\rm coeff}_E(B_Y) 
\text{ where }
K_Y+B_Y=\pi^*(K_X+B). 
\]
The log discrepancy $a_E(X,B)$ depends only on the divisorial valuation induced by $E$ on $X$
and not in the model $Y$.
We say that a log pair $(X,B)$ is {\em log canonical} (resp. {\em Kawamata log terminal}) if all its log discrepancies are non-negative (resp. positive).
We abbreviate log canonical (resp. Kawamata log terminal) by lc (resp. klt).
A {\em non-klt place} of $(X,B)$ is a prime divisor $E$ over $X$
for which $a_E(X,B)\leq 0$.
If $(X,B)$ is lc, then the non-klt places are also called log canonical places.
A {\em log canonical center} of $(X,B)$ 
is the image on $X$ of a log canonical place.
}
\end{definition}

\begin{definition}
{\em 
Let $(X,B)$ be a log pair.
We say that $(X,B)$ is {\em canonical} if all its log discrepancies are at least $1$.
A divisor $E$ over $X$ for which $a_E(X,B)=1$ holds is called a {\em canonical place} of the log pair $(X,B)$.
We adopt the previous terminology regardless of whether $(X,B)$ is a canonical pair or not.
}
\end{definition}

\begin{definition}
{\em 
Let $(X,B)$ be a log canonical pair.
We say that $(X,B)$ is {\em divisorially log terminal} (or dlt for short) if there exists an open subset $U\subset X$ satisfying the following conditions:
\begin{enumerate}
    \item the pair $(U,B_U)$ is log smooth,
    \item every log canonical center of $(X,B)$
    intersects $U$ and is a stratum of
    $\lfloor B_U\rfloor$.
\end{enumerate}
}
\end{definition}

\begin{definition}
{\em 
Let $(X,B)$ be a log pair. 
A {\em dlt modification} $p\colon Y\rightarrow X$
of $(X,B)$ is a projective birational morphism
satisfying the following conditions:
\begin{enumerate}
    \item the pair $(Y,B_Y)$ is dlt, where $B_Y={\rm Ex}(p)+p^{-1}_*B$, 
    \item every $p$-exceptional divisor is a non-klt place of $(X,B)$, and 
    \item we have that $K_Y+B_Y$ is nef over $X$, 
\end{enumerate}
In the previous setting, we may also say that
$p\colon (Y,B_Y)\rightarrow (X,B)$ is a {\em dlt modification}.
}
\end{definition}

We recall the following lemma (see, e.g.,~\cite[Theorem 2.9]{FS20}).

\begin{lemma}\label{lem:existence-dlt-mod}
Let $(X,B)$ be a log pair. 
Then, there exists a dlt modification 
$p\colon (Y,B_Y)\rightarrow (X,B)$.
\end{lemma}

\begin{definition}\label{def:dlt-stratum}
{\em 
Let $(X,B)$ be a log pair.
A {\em dlt stratum} of $(X,B)$ is either:
\begin{itemize}
\item a stratum of $\lfloor B_Y\rfloor$, or \item the variety $Y$, 
\end{itemize}
where $(Y,B_Y)\rightarrow (X,B)$ is a dlt modification of $(X,B)$.
We will refer to the collection of all dlt strata among all dlt modifications of $(X,B)$ as the {\em dlt strata} of $(X,B).$
}
\end{definition} 

We will need the following lemma. 

\begin{lemma}\label{lem:dlt-mod-non-lc}
Let $(X,B)$ be a log canonical pair.
Let $P$ be a divisor on $X$.
For $\epsilon>0$ small enough
a dlt modification of $(X,B+\epsilon P)$
only extracts log canonical places of $(X,B)$.
\end{lemma}

\begin{proof}
Let $q\colon Z\rightarrow X$ be a log resolution
of $(X,B+P)$.
Log discrepancies are continuous with respect to the boundary divisor.
Hence, for $\epsilon>0$ small enough, for every $q$-exceptional divisor $E$,
$a_E(X,B+\epsilon P)<0$ if and only if 
$a_E(X,B)=0$.
Thus, every non-klt place of $(X,B+\epsilon P)$
is a log canonical place of $(X,B)$.
\end{proof}

\begin{definition}
{\em 
A projective log pair $(X,B)$ is said to be {\em log Calabi--Yau} if $K_X+B\equiv 0$ and $(X,B)$ is lc.
}
\end{definition}

For a log Calabi--Yau pair $(X,B)$ 
it is known that $K_X+B\sim_\qq 0$ (see, e.g.,~\cite[Theorem 1.5]{FG14}).

\begin{definition}
{\em 
The {\em index} of a log Calabi--Yau pair $(X,B)$ is the smallest positive integer $m$ for which $m(K_X+B)\sim 0$.
}
\end{definition}

\begin{definition}
{\em 
Let $(X,B)$ and $(Y,B_Y)$ be two log pairs
and let $\pi\colon X\dashrightarrow Y$ be a birational map.
We say that $\pi$ is {\em crepant} for $(X,B)$ and $(Y,B_Y)$ if the following condition holds.
There exists a common resolution
$p\colon Z\rightarrow X$
and 
$q\colon Z\rightarrow Y$
for which $p^*(K_X+B)=q^*(K_Y+B_Y)$.
In this case, we also say that
$\pi\colon (X,B)\rightarrow (Y,B_Y)$ is a {\em crepant birational map}
and that
$(Y,B_Y)$ is {\em crepant birational equivalent} to $(X,B)$.
In the previous setting, we write
$(X,B)\simeq_{\rm bir} (Y,B_Y)$.
}
\end{definition}

If $(X,B)$ is a log Calabi--Yau pair,
then every log pair $(Y,B_Y)$ which is crepant birational equivalent to $(X,B)$ is also log Calabi--Yau. Furthermore, they have the same index (see, e.g.~\cite[Lemma 3.1]{FMM22}).

\begin{definition}
{\em 
Let $(X,B)$ be a log Calabi--Yau pair of dimension $n$.
The {\em coregularity} of $(X,B)$, denoted by ${\rm coreg}(X,B)$, is the dimension of the smallest log canonical center in any dlt modification of $(X,B)$.
If $(X,B)$ is klt, then we set ${\rm coreg}(X,B)=n$.
Thus, we have 
${\rm coreg}(X,B)\in \{0,\dots,n\}$.
}
\end{definition}

We finish this subsection by collecting some lemmata.
The following lemma is a straightforward application of the connectedness of log canonical centers, see, e.g.,~\cite[Theorem 1.2]{FS20}.

\begin{lemma}\label{lem:contraction-S^1}
Let $(X,B)$ be a log Calabi--Yau surface with ${\rm coreg}(X,B)=0$
and $K_X+B\sim 0$.
Let $p\colon X\rightarrow Y$ be a projective birational morphism contracting the irreducible curve $C$. 
Assume that the curve $C$ is not contained in $B$.
Then, the curve $C$ intersects $B$ in at most one point.
\end{lemma}

The following lemma is known to the experts.

\begin{lemma}\label{lem:factoring-contraction}
    Let $(X,B)$ be a log Calabi--Yau surface of index one. Let $f\colon X\rightarrow Y$ be a birational contraction. Then we can factor $f$ as $h\circ g$ where $g$ is a birational contraction that only contracts canonical places of $(X,B)$ and $h$ is a birational contraction that only contracts log canonical places of $(X,B)$.
\end{lemma}

\subsection{Birational complexity}
In this subsection, we recall the concepts of complexity and birational complexity.

\begin{definition}
{\em 
Let $(X,B)$ be a log pair.
The {\em complexity} of $(X,B)$,
denoted by $c(X,B)$, is defined to be
\[
\dim X + \dim {\rm Cl}_\qq(X) - |B|,
\]
where $|B|$ is the sum of the coefficients of $B$.
}
\end{definition}

The complexity of a log Calabi--Yau pair
is non-negative. 
Furthermore, if the complexity is zero, then
$(X,\lfloor B\rfloor)$ is toric (see~\cite[Theorem 1.2]{BMSZ18}).

\begin{definition}
{\em 
Let $(X,B)$ be a log Calabi--Yau pair. 
The {\em birational complexity} of $(X,B)$,
denoted by $c_{\rm bir}(X,B)$, is defined to be:
\[
\min\{ c(Y,B_Y) \mid 
(Y,B_Y)\simeq_{\rm bir}(X,B) 
\}.
\]
}
\end{definition}

The minimum in the previous definition exists due to \cite[Lemma 2.31]{MM24}.

The following theorem is due to Mauri and the third author (see~\cite[Theorem 1.6]{MM24}).

\begin{theorem}
Let $(X,B)$ be a log Calabi--Yau pair
of dimension $n$ and index one.
We have $c_{\rm bir}(X,B)=0$
if and only if
$(X,B)\simeq_{\rm bir}(\pp^n,\Sigma^n)$.
\end{theorem}

The previous theorem states that
a log Calabi--Yau pair
admits a birational toric model 
if and only if
it has birational complexity zero.

\subsection{Torus exceptional degree}
\label{subsec:ted}
In this subsection, we introduce the concept of torus exceptional degree and related invariants. 

\begin{definition}\label{def:ted}
{\em 
Let $(X,B)$ be a log Calabi--Yau pair.
Let $\phi\colon (\pp^n,\Sigma^n)\dashrightarrow (X,B)$
be a crepant birational map.
The {\em torus exceptional degree}
of $\phi$ is defined to be
\[
\mathcal{T}(\phi):=
\deg_{\pp^n}\left(
\overline{{\rm Ex}^1(\phi)|_{\mathbb{G}_m^n}}
\right) 
\]
where ${\rm Ex}^1(\phi)$ denotes
the reduced sum of the exceptional divisors of $\phi$.
The torus exceptional degree is a non-negative integer.
}
\end{definition}

For instance, for every toric birational map
$\phi\colon (\pp^n,\Sigma^n)\dashrightarrow (T,B_T)$ we have $\mathcal{T}(\phi)=0$ although the exceptional locus of $\phi$ may contain the hyperplane coordinates as exceptional divisors. Indeed, in this case, we have that $\phi|_{\mathbb{G}_m^n}$ is an isomorphism.

\begin{definition}
{\em 
Let $(X,B)$ be a log Calabi--Yau pair.
The {\em torus exceptional degree} of $(X,B)$ is defined to be
\[
\mathcal{T}(X,B):=\min\{\mathcal{T}(\phi)\mid 
\phi\colon (\pp^n,\Sigma^n)\dashrightarrow (X,B)
\text{ is a crepant birational map}
\}.
\]
If the set in the previous definition is empty,
then we set $\mathcal{T}(X,B)=\infty$.
}
\end{definition}

Given a crepant birational map
$\phi\colon (\pp^n,\Sigma^n)\dashrightarrow (X,B)$, we can produce crepant birational maps to $(X,B)$
of arbitrarily large torus exceptional degree by precomposing with Cremona transformations. 
On the other hand, finding crepant birational maps 
with small torus exceptional degree is related to 
proving affine boundedness of the pairs $(X,B)$.
This is why we are mostly concerned with minimizing this invariant. 
The torus exceptional degree $\mathcal{T}(\phi)$ is minimized precisely when $X\setminus B$ admits an affine open which is isomorphic to a big open subset of an algebraic torus. This motivates the following definition.

\begin{definition}\label{def:almost-tori}
{\em 
An {\em almost torus}
is a big open subset of an algebraic torus $\mathbb{G}_m^n$.
}
\end{definition}

If $(X,B)$ is a log Calabi--Yau pair with
$\mathcal{T}(X,B)=0$, then 
there is an open embedding
$U\hookrightarrow X\setminus B$
of an almost torus $U$.
Now, we turn to introduce some definitions 
and invariants 
that allow us to understand how far is
$X\setminus B$ from being covered by almost tori.

\begin{definition}\label{def:div-cover}
{\em
Let $X$ be an algebraic variety.
Let $\{U_i\}_{i\in I}$ be a finite set of 
 open subvarieties $U_i\subset X$.
We say that the set $\{U_i\}_{i\in I}$ {\em divisorially covers} $X$ if 
\[
{\rm codim}_X(X\setminus\cup_{i\in I}U_i)\geq 2.
\]
}
\end{definition}

It is clear that a set that divisorially covers $X$ admits a finite subset that divisorially covers $X$.

\begin{definition}\label{def:div-cover-logCY}
{\em 
Let $(X,B)$ be a log Calabi--Yau pair.
Let $\{\phi_i\}_{i\in I}$
be a set 
of crepant birational maps 
$\phi_i\colon (\pp^n,\Sigma^n)\dashrightarrow (X,B)$.
We say that the set
$\{\phi_i\}_{i\in I}$
{\em divisorially covers} $(X,B)$ 
if the open subvarieties 
$\phi_i(\mathbb{G}_m^n\setminus {\rm Ex}(\phi_i))$ divisorially cover $X\setminus B$.
In the previous setting, we also say that
$(X,B)$ is {\em divisorially covered}
by the crepant birational maps
$\{\phi_i\}_{i\in I}$.
}
\end{definition}

\begin{definition}
{\em 
Let $(X,B)$ be a log Calabi--Yau pair.
The {\em total torus exceptional degree} 
is defined to be
\[
\mathcal{T}_t(X,B):=
\min
\left\{ 
\sum_{i\in I} \mathcal{T}(\phi) 
\,\middle|\,
\text{
the crepant birational maps
$\{\phi_i\}_{i\in I}$ 
divisorially cover $(X,B)$
}
\right\}.
\]
}    
\end{definition}

The following lemma will be used in Example~\ref{ex:total-ted-unbounded} to show that there are examples of log Calabi--Yau surfaces $(X,B)$ for which $\mathcal{T}_t(X,B)$ is finite but arbitrarily large.

\begin{lemma}\label{lem:rank-inf}
Let $X$ be a rational projective variety.
Let $\{D_i\}_{i\in \nn}$ be an infinite sequence of prime divisors. Set
\[
r_k:=\min 
\left\{ 
\dim H_1\left(X^{\rm reg}\setminus \cup_{i\in I} D_i,\qq \right) 
\mid 
I\subset \nn 
\text{ and }
|I|=k 
\right\}. 
\]
Then, we have $\lim_{k\rightarrow \infty}r_k=\infty$.
\end{lemma}

\begin{proof}
Let $\phi\colon \pp^n\dashrightarrow X$ be a birational map.
We can find a hypersurface $H\subset \pp^n$ such that $\phi$ restricts to an embedding
$\pp^n \setminus H\hookrightarrow X$.
Let $P_1,\dots,P_\ell$ be the prime components of 
$X^{\rm reg}\setminus \phi(\pp^n\setminus H)$. 
Set 
\[
s_k:=\min \left\{ 
\dim H_1 \left( 
(\pp^n \setminus H)\setminus 
\cup_{i\in I} \phi^*D_i, \qq \right)  
\mid 
I\subset \nn 
\text{ and }
|I| = k
\right\}.
\]
By~\cite[Proposition 4.1.3]{Dim92}, we know that $\lim_{k\rightarrow \infty} s_k =\infty$.
On the other hand, for every finite subset 
$I\subset \nn$, we have 
\[
\dim H_1\left( 
X^{\rm reg}\setminus \cup_{i\in I} D_i, \qq
\right) 
- 
\dim H_1 \left( 
(\pp^n \setminus H)\setminus 
\cup_{i\in I} \phi^*D_i, \qq
\right)  
\leq \ell. 
\]
Then, the statement follows.
\end{proof}

\subsection{Boundedness}
In this subsection, we prove the following lemma related to affinely bounded families of normal varieties.

\begin{lemma}\label{lem:affine-bounded-pi_1}
Let $\mathcal{F}$ be a family of normal varieties which is affinely bounded.
Then, there exists a constant $r:=r(\mathcal{F})$, only depending on $\mathcal{F}$, satisfying the following.
Let $X\in \mathcal{F}$, then $d(\pi_1(X^{\rm reg}))\leq r$.
\end{lemma} 

\begin{proof}
Let $\mathcal{A}\rightarrow T$ be a bounded family of affine varieties such that for every $X\in \mathcal{F}$ there is an open embedding
$\mathcal{A}_t\hookrightarrow X$.
Then, the family $\mathcal{A}^{\rm reg}:=\{U^{\rm reg} \mid U\in \mathcal{A}\}$ is also bounded.
Let $\mathcal{A}^{\rm reg} \rightarrow W$ be its bounding family.
By Verdier Theorem~\cite[Corollaire 5.1]{Ver76}, we can stratify the morphism $\mathcal{A}^{\rm reg}\rightarrow W$ into finitely many locally trivial fibrations 
$\mathcal{A}^{\rm reg}_i \rightarrow W_i$.
We conclude that the family of groups
$\{ \pi_1(\mathcal{A}^{\rm reg}_t)/ \simeq\}_{ t\in T}$ is a finite set.

In particular, $d(\pi_1(\mathcal{A}^{\rm reg}_t))$ is bounded above in terms of $\mathcal{F}$, meaning that $d(\pi_1(\mathcal{A}^{\rm reg}_t))\leq r(\mathcal{F})$ holds for every $t\in T$.
Let $X\in \mathcal{F}$. Then there is an open embedding $\mathcal{A}_t \hookrightarrow X$ for some $t\in T$.
Therefore, we have a surjective homomorphism
\[
\pi_1(\mathcal{A}^{\rm reg}_t) \rightarrow \pi_1(X^{\rm reg}).
\]
Thus, we have $d(\pi_1(X^{\rm reg}))\leq r(\mathcal{F})$.
\end{proof}

The previous lemma will be used in Example~\ref{ex:bir-bound-not-aff-bound}
to show an example of a birationally
bounded family
that is not affinely bounded.

\subsection{Cluster type}
In this subsection, we introduce the concept of cluster type log Calabi--Yau pairs and prove a lemma.

\begin{definition}
{\em  
Let $\phi\colon (\pp^n,\Sigma^n)\dashrightarrow(X,B)$ be a crepant birational map.
We say that $\phi$ is {\em of cluster type}
if $\mathcal{T}(\phi)=0$.
Equivalently, a crepant birational map
is said to be of cluster type
if it induces an embedding of an almost torus
in $X\setminus B$.
}
\end{definition}

\begin{definition}
{\em 
Let $(X,B)$ be a log Calabi--Yau pair.
We say that $(X,B)$ is {\em of cluster type}
if $\mathcal{T}(X,B)=0$.
In other words, a log Calabi--Yau pair
$(X,B)$ is said to be 
{\em of cluster type} 
if there exists 
a crepant birational map
$\phi\colon (\pp^n,\Sigma^n)\dashrightarrow(X,B)$
which is of cluster type.
}
\end{definition}

\begin{definition}
{\em 
Let $(X,B)$ be a log Calabi--Yau pair.
The {\em cluster type} of $(X,B)$, denoted by $\mathcal{C}(X,B)$, is the least number of almost tori that divisorially cover $X\setminus B$.
If $X\setminus B$ is not divisorially covered by almost tori, then we set $\mathcal{C}(X,B)=\infty$.
}    
\end{definition}

Following the previous definitions, we may say that a log Calabi--Yau pair $(X,B)$ is of cluster type $\mathcal{C}(X,B)$.
Theorem~\ref{introthm:t-equal-zero} states that a log Calabi--Yau pair $(X,B)$ is of cluster type 
if and only if $X\setminus B$ 
is divisorially covered by finitely many almost tori.

\begin{lemma}\label{lem:diagram-cluster-type-surf}
Let $(X,B)$ be a log Calabi--Yau surface.
Let $\phi \colon (\pp^2,\Sigma^2)\dashrightarrow (X,B)$ be a cluster type crepant birational map.
Then, there is a commutative diagram of projective birational maps:
\[
\xymatrix{
(T,B_T)\ar[d]_-{p_1} & (X_0,B_0)\ar[l]^-{q}\ar[d]^-{p_2} \\
(\pp^2,\Sigma^2) & (X,B) \ar@{-->}[l]^-{\phi^{-1}}
}
\] 
where the following conditions are satisfied:
\begin{itemize}
    \item $p_1$ is a toric projective birational morphism,
    \item $p_2$ is a projective birational morphism that only extracts log canonical places of $(X,B)$, and 
    \item $q$ is a projective birational morphism that only extracts canonical places of $(T,B_T)$.
\end{itemize}
\end{lemma}

\begin{proof}
We can find a projective birational morphism 
$(X_0,B_0)\rightarrow (X,B)$ only extracting log canonical places of $(X,B)$ for which the birational contraction $X_0\rightarrow \pp^2$ is a morphism. Indeed, by definition of cluster type birational map any canonical place of $(X,B)$ which is exceptional over $X$ must also be exceptional over $\pp^2.$ Observe that the morphism $X_0\rightarrow \pp^2$ is the outcome of the $K_{X_0}$-MMP over $\pp^2$.
By Lemma~\ref{lem:factoring-contraction}, we can factorize $X_0\rightarrow \pp^2$ as a contraction of canonical places of $(X,B)$ 
followed by a contraction of log canonical places of $(X,B)$.
Since $T\rightarrow \pp^2$ only extracts log canonical places of $(\pp^2,\Sigma^2)$, we conclude that $p_1$ is a toric birational morphism.
\end{proof}

\section{Birational complexity zero}
In this section, we show a structural theorem for varieties of birational complexity zero.
In the first subsection, we study affine varieties that divisorially cover log Calabi--Yau pairs of birational complexity zero.
In the second subsection, we study the birational complexity of Gorenstein del Pezzo surfaces of rank one.

\subsection{Log Calabi--Yau pairs.}
In this subsection, we give a proof of Theorem~\ref{introthm:covering-bcomp0}.
First, we will need the following lemma.

\begin{lemma}\label{lem:special-model-complexity-1}
Let $(X,B)$ be a toric log Calabi--Yau pair of dimension $n$.
Let $E$ be a divisorial valuation exceptional over $X$ with $a_E(X,B)=1$.
Then, there exists a crepant birational map 
$\phi\colon (Y,B_Y)\dashrightarrow (X,B)$ satisfying the following conditions:
\begin{enumerate}
    \item we have $c(Y,B_Y)=1$, 
    \item the center of $E$ on $Y$ is a prime divisor $E_Y$, 
    \item the variety $Y$ is $\qq$-factorial and admits a fibration $p\colon Y\rightarrow Z$ 
    to a variety of dimension $n-1$,
    \item there is a prime component $F_Y$ of 
    $p^{-1}(p(E_Y))$ that dominates $p(E_Y)$, is not contained in $B_Y$, and it is different from $E_Y$, and
    \item the pair $(Y,B_Y+\epsilon F_Y)$ is dlt for $\epsilon>0$ small enough.
\end{enumerate}
Furthermore, the birational map $\phi\colon Y\dashrightarrow X$ only extracts $E_Y$ and log canonical places of $(X,B)$.
\end{lemma}

\begin{proof}
First, we argue that in some log resolution $(T,B_T)$ of $(X,B)$ the center of $E$ on $T$ is a prime divisor of a prime component of $B_T$. Since $\mathbb{G}_m^n$ is smooth, for every toric modification
$(X_i,B_i)\rightarrow (X,B)$ the center of $E$ on $X_i$ must be contained in $B_i$.
We construct a sequence of toric birational morphisms as follows: we set $(X_0,B_0):=(X,B)$ and for each $i$ we let $(X_i,B_i)$ to be a toric log resolution of a blow up
of the smallest stratum of $B_i$ containing $c_{E}(X_i)$. This process stops precisely when $c_{E}(X_i)$ is contained in a prime component of $B_i$ and is not contained in smaller dimensional strata.
For a fixed constant $k$, there are only finitely many toric valuations $F$ over $X$ with $a_F(X)<k$. 
Thus, for $i$ large enough, we
have $a_F(X)\geq k$ 
for every prime divisor $F$ extracted by $X_{i+1}\rightarrow X_i$.
By~\cite[Lemma 2.29]{KM98}, we have 
\[
a_E(X) \geq \min \{ a_P(X) \mid 
\text{ 
$P\subset B_i$ prime and $c_E(X_i)\subset P$
}
\}.
\]
We conclude that this process must stop.
Thus, we can find a log resolution $(T,B_T)$ of $(X,B)$ where the center of $E$ is a prime divisor of a component of $B_T$.
By~\cite[Theorem 1.2]{Tev07}, we can further assume that the center $c_E(T)$ contains no other strata.
Note that $(T,B_T)\rightarrow (X,B)$ is a toric birational map so $T\rightarrow X$ only extracts log canonical places of $(X,B)$.

Let $\Sigma_T \subset N_\qq$ be the fan associated to $T$.
Let $Q\subset B_T$ be the unique prime component that contains $c_E(T)$.
Let $\rho_Q\in \Sigma_T(1)$ be the ray corresponding to the torus invariant divisor $Q$.
Let $N_{0,\qq}\subset N_\qq$ be an orthogonal complementary subspace to the ray generated by $\rho_Q$. Let $\{e_1,\dots,e_{n-1}\}$ be an orthogonal basis of $N_{0,\qq}$.
Consider the fan 
\[
\Sigma_0:=\{\pm e_1,\dots, \pm e_{n-1}, \pm \rho_Q\}.
\]
Then $X(\Sigma_0)\simeq (\pp^1)^n$, 
the center of $Q$ on $X(\Sigma_0)$ is a torus invariant divisor, and 
this torus invariant divisor is a section 
for a fibration $(\pp^1)^n\rightarrow (\pp^1)^{n-1}$.
Let $\Sigma_1$ be a common refinement of $\Sigma_T$ and $\Sigma_0$.
Let $T_1$ be the projective toric variety associated to $\Sigma_1$.
Then, we have a projective birational morphism 
$T_1\rightarrow T$
and a fibration $p_{T_1}\colon T_1\rightarrow Z:=(\pp^1)^{n-1}$
such that the strict transform of $Q$ on $T_1$ dominates $Z$.
We let $(T_1,B_{T_1})$ be the log pull-back of $(T,B_T)$ to $T_1$.
Observe that $T_1\rightarrow X$ is a toric birational morphism,
so it only extracts log canonical places of $(X,B)$.

Let $Y_0\rightarrow T_1$ be a log resolution of $(T_1,B_{T_1})$ that extracts $E$.
Let $E_{Y_0}$ be the center of $E$ on $Y_0$.
Let $B_{Y_0}$ be the strict transform of $B_{T_1}$ on $Y_0$.
Let $R_{Y_0}$ be the reduced exceptional of $Y_0\rightarrow T_1$ minus $E$.
Then, the pair $(Y_0,B_{Y_0}+R_{Y_0})$ is log smooth.
Note that 
\[
K_{Y_0}+B_{Y_0}+R_{Y_0} \sim_\qq 
\sum_{\substack{P\subseteq Y_0 \\ P\neq E_{Y_0}}}
a_{P}(T_1,B_{T_1})P.
\]
We call $\Omega_{Y_0}$ the effective divisor on the right side of the previous $\qq$-linear equivalence.
Note that $\Omega_{Y_0}$ contains the support of all the prime divisors exceptional over $T_1$ that have positive log discrepancy with respect to $(T_1,B_{T_1})$ except for $E_{Y_0}$.
The divisor $\Omega_{Y_0}$ is degenerate over $T_1$.
We run a $(K_{Y_0}+B_{Y_0}+R_{Y_0})$-MMP over $T_1$ with scaling of an ample divisor.
After finitely many steps all the prime components of $\Omega_{Y_0}$ are contracted. 
Note that this MMP terminates after these components are contracted.
We let $Y_1$ be the model obtained by this MMP.
Let $B_{Y_1}$ be the push-forward of $B_{Y_0}$ on $Y_1$.
Let $E_{Y_1}$ be the push-forward of 
$E_{Y_0}$ on $Y_1$.
Then, the pair $(Y_1,B_{Y_1})$ satisfies the following conditions:
\begin{enumerate}
\item[(i)] the pair $(Y_1,B_{Y_1})$ is $\qq$-factorial and dlt, and 
\item[(ii)] the birational morphism $Y_1\rightarrow T_1$ only extracts $E_{Y_1}$ and log canonical places of $(T_1,B_{T_1})$.
\end{enumerate}
Condition (ii) implies that $c(Y_1,B_{Y_1})=1$.
Let $p_{Y_1}\colon Y_1\rightarrow Z$ be the induced fibration and let $E_{Y_1}$ be the center of $E$ on $Y_1$.
Then we have that $p_{Y_1}^{-1}(p_{Y_1}(E_{Y_1}))$ has at least two prime components: $E_{Y_1}$ and the strict transform of $p_{T_1}^{-1}(p_{Y_1}(E_{Y_1}))$.
Thus, the log pair $(Y_1,B_{Y_1})$ satisfies conditions (1)-(4).
Let $F_{Y_1}$ be a prime component of $p_{Y_1}^{-1}(p_{Y_1}(E_{Y_1}))$ that dominates $p_{Y_1}(E_{Y_1})$, is not contained in $B_{Y_1}$, and is different from $E_{Y_1}$.
By~\cite[Theorem 3.1]{KK10}, we can take a minimal dlt modification $(Y,B_Y+\epsilon F_Y)$ of 
$(Y_1,B_{Y_1}+\epsilon F_{Y_1})$.
By Lemma ~\ref{lem:dlt-mod-non-lc}, for $\epsilon>0$ small enough all the non-klt places
of $(Y_1,B_{Y_1}+\epsilon F_{Y_1})$ are log canonical places of $(Y_1,B_{Y_1})$.
In particular, the pair $(Y,B_Y)$ is a dlt modification of $(Y_1,B_{Y_1})$.
Thus, we have $c(Y,B_Y)=1$.
Conditions (1)-(4) are clearly preserved on $(Y,B_Y)$.
We conclude that the pair $(Y,B_Y)$ satisfies the conditions (1)-(5).
Note that the birational map
$\phi\colon Y\dashrightarrow X$ only extracts $E_Y$ and log canonical places of $(X,B)$.
\end{proof}

Now, we turn to prove the first theorem of this article.

\begin{proof}[Proof of Theorem~\ref{introthm:covering-bcomp0}]
Let $(X,B)$ be a log Calabi--Yau pair of dimension $n$, index one, and birational complexity zero.
By~\cite[Theorem 1.6]{MM24}, there is a crepant birational map $\phi\colon (T,B_T)\dashrightarrow (X,B)$
where $(T,B_T)$ is a toric log Calabi--Yau pair.
Let 
\[
U_1:=\phi(T\setminus (B_T\cup {\rm Ex}(\phi))).
\]
Then $U_1$ is isomorphic to an open subvariety of $T\setminus B_T\simeq \mathbb{G}_m^n$.
As $(T,B_T)$ has no log canonical centers along $T\setminus B_T$, we conclude that 
$\phi$ induces an embedding $U_1\hookrightarrow X\setminus B$.
Let $E_2,\dots,E_k$ be the $(n-1)$-dimensional prime components of $(X\setminus B)\setminus U_1$.

For each $i\in \{2,\dots,k\}$, we argue that there exists a crepant birational map
$\phi_i\colon (T_i,B_{T_i}) \dashrightarrow (X,B)$,
from a toric log Calabi--Yau pair $(T_i,B_{T_i})$ such that the center of $E_i$ on $T_i$ is a divisor intersecting $T_i\setminus B_{T_i}$. 
This implies that the divisor $E_i$ intersects
\[
U_i:=\phi_i(T_i\setminus (B_{T_i}\cup {\rm Ex}(\phi_i))). 
\]
Thus, the open affines $U_1,\dots,U_k$ cover $X\setminus B$ divisorially.

It suffices to show that the statement holds for $E:=E_2$.
By assumption, the divisor $E$ gives a divisorial valuation exceptional over $T$. 
Note that $a_E(T,B_T)=1$.
We can apply Lemma~\ref{lem:special-model-complexity-1} to the toric log Calabi--Yau pair $(T,B_T)$ and the valuation $E$. 
We obtain a crepant birational map
$\phi\colon (Y,B_Y)\dashrightarrow (T,B_T)$
satisfying the following conditions:
\begin{enumerate}
\item we have $c(Y,B_Y)=1$,
\item the center of $E$ on $Y$ is a prime divisor $E_Y$, 
\item the variety $Y$ is $\qq$-factorial and admits a fibration $p\colon Y\rightarrow Z$ to a variety of dimension $n-1$,
\item there is a prime component $F_Y$ of $p^{-1}(p(E_Y))$ that dominates $E_Y$, is not contained in $B_Y$, and is different from $E_Y$, and 
\item the pair $(Y,B_Y+\epsilon F_Y)$ is dlt for $\epsilon>0$ small enough.
\end{enumerate}
By condition (5), we can run a $(K_Y+B_Y+\epsilon F_Y)$-MMP over $Z$ with scaling of an ample divisor.
Note that $K_Y+B_Y\sim_\qq 0$ so 
$K_Y+B_Y+\epsilon F_Y\sim_\qq \epsilon F_Y$.
By condition (4), we know that the divisor $F_Y$ is degenerate over $Z$, so this MMP terminates after contracting $F_Y$.
Let $T_1$ be the model obtained after contracting $F_Y$. Let $B_{T_1}$ be the push-forward of $B_Y$ on $T_1$.
Since $F_Y$ is not contained in $B_Y$, the complexity drops precisely by one when we contract this divisor.
Hence, we get $c(T_1,B_{T_1})=0$, so $(T_1,B_{T_1})$ is a toric log Calabi--Yau pair.
Further, the center of $E$ on $T_1$ is a divisor that intersects $T_1\setminus B_{T_1}$. This finishes the proof.
\end{proof}

\subsection{Log Calabi--Yau Surfaces.}
In this subsection we study the birational complexity log Calabi--Yau pairs $(X,B)$, where $X$ is a Goresnstein del Pezzo surface of Picard rank $1$.

\begin{lemma}\label{lem:1-comp-DE}
Let $(X;x)$ be a $D$-type or $E$-type singularity.
Then the only $1$-complement of $(X;x)$ is the trivial one, i.e., $B=0$.
\end{lemma}

\begin{proof}
Let $p\colon Y\rightarrow X$ be a minimal resolution.
For each $E\subset Y$ exceptional over $X$ we have
$a_E(X)=1$.
Assume that $(X,B;x)$ is a $1$-complement and $B\neq 0$ through $x\in X$.
Then $a_E(X,B)$ is integral and strictly less than $1$ for each $E\subset Y$ exceptional over $X$.
Thus, $a_E(X,B)=0$ for every such a curve.
We conclude that $p^*(K_X+B)=K_Y+B_Y+E$ where $E$ is the reduced exceptional divisor.
Here, $B_Y$ is the strict transform of $B$ on $Y$.
As $(X;x)$ is either $D$-type or $E$-type, 
there is a curve $C\subset E$ that intersects at least $3$ other prime components $E_1,E_2$, and $E_3$ of $E$.
Note that $K_Y+B_Y+E$ is log canonical and $\qq$-trivial over $X$.
Performing adjunction to $C$ we observe 
\[
(K_Y+B_Y+E)|_C=K_C+E_1|_C+E_2|_C+E_3|_C+B_Y|_C.
\]
From the previous formula, we get that
$(K_Y+B_Y+E)\cdot C>0$ leading to a contradiction.
\end{proof}

By~\cite[Theorem 1.2]{Ye02}, there is an explicit list of the possible singularities Gorenstein del Pezzo surfaces of Picard rank $1$. Moreover these are up to isomorphism uniquely determined by their singularity type, except for the cases with singularities in $\{E_8,E_7+A_1,E_6+A_2,D_4+D_4\}$. 

For the case of singularities $D_4+D_4$, there is an infinite family of Gorenstein del Pezzo surfaces of Picard rank $1$, we will label any member of this family by $X(D_4+D_4)$.

For each of the singularities $\{E_8,E_7+A_1,E_6+A_2\}$, there are two isomorphism classes of Gorenstein del Pezzo surfaces of Picard rank $1$, and they can be obtained as contractions from rational elliptic surfaces. For each one of these singularities there is a case where the elliptic surface does not have an $I_1$ singular fiber (see \cite[Table 4.1]{BBD84}), these cases will be labeled $X_1(-)$. The other case will be labeled $X(-)$ as with any other singularity that has a unique member.

The main statement of this subsection is the following theorem:

\begin{theorem}\label{thm:bir-comp-Del-Pezzo-rank-1}
Let $X$ be a Gorenstein del Pezzo surface of rank $1$. Then there exists a $1$-complement $B$ of $X$ with $c_{\rm bir}(X,B)=0$, if and only if $X$ is not isomorphic to $X(D_4+D_4), X_1(E_8), X_1(E_7+A_1), X_1(E_6+A_2)$. 
\end{theorem}

In \cite[Example 3.13]{Mor22}, it is shown that the varieties of type $X(D_4+D_4)$ do not admit a $1$-complement of coregularity $0$, hence a fortiori they cannot have a $1$-complement with birational complexity $0$. This is also true for the remaining $3$ varieties in the statement of Theorem~\ref{thm:bir-comp-Del-Pezzo-rank-1}

\begin{lemma} \label{lem:no-one-complement}
    Let $X$ be one of $X_1(E_8), X_1(E_7+A_1), X_1(E_6+A_2)$, then $X$ admits no index one complement of coregularity $0$.
\end{lemma}

\begin{proof}
    Assume that $B$ is such that $K_X+B\sim 0$ and ${\rm coreg}(X,B)=0$. Let $x\in X$ be the $E$-type singularity, and $f:Y \rightarrow X$ the minimal resolution of $X$. Let $B_Y$ be the strict transform of $B$. 

    Let $E_1,E_2,\ldots E_k$ be the exceptional curves over $x$ of $f:Y\rightarrow X$, where $E_2, \ldots E_k$ form a chain, and $E_1$ intersects only $E_4$. Here $k$ can be $6$, $7$ or $8$. By Lemma~\ref{lem:1-comp-DE}, $B$ cannot contain $x$, therefore $B_Y$ cannot intersect any $E_i$.
    By \cite[Table II]{BBD84}, there exists a unique sequence of blow-downs $\varphi:Y \rightarrow \pp^2$ that contracts every $E_i$, except for $E_1$. Let $E_1'$ and $B'$ be the images of $E_1$ and $B_Y$ by $\varphi$. In $\pp^2$, $B'$ is of degree $3$ and $E_1'$ is a line.

    By the description of the maps in \cite[Table II]{BBD84}, in $\varphi: Y \rightarrow  \pp^2$ there is a unique point in $E_1'$ over which blow ups are performed, three of these are performed over the strict transforms of $E_1'$, whose exceptional divisors are the strict transforms of $E_2$, $E_3$ and $E_4$.
    If any irreducible component of $B'$ were to have order of contact with $E_1'$ of $1$ or $2$, it would imply that $B_Y$ would intersect $E_2$ or $E_3$ respectively. Therefore $B'$ can have only one irreducible component, hence $B$ is also irreducible. Since ${\rm coreg}(X,B)=0$, $B$ must have nodal singularities. This would imply that a fiber in the rational elliptic surface from which $X$ can be obtained (see \cite[Table 4.1]{BBD84}) is a nodal curve, a contradiction

\end{proof}

To prove Theorem~\ref{thm:bir-comp-Del-Pezzo-rank-1}, we will study some special configurations of curves in $\mathbb{P}^2$ that get mapped to the exceptional divisors of minimal resolutions of these surfaces.

\begin{lemma}\label{lem:Gor-del-Pezzo-cubic}
Let $X$ be a Gorenstein del Pezzo surface of Picard rank one, whose singularity type is $\{A_2+A_1,A_4,D_5,E_6,E_7,E_8,D_6+A_1,D_8,E_7+A_1,A_2+A_2+A_2,A_2+A_5,A_2+E_6,D_4+3A_1, 2A_1+D_6,A_1+A_3+A_3,D_5+A_3,A_2+A_2+A_2+A_2\}$ and let $(\mathbb{P}^2,C)$ be the projective plane with a nodal cubic. Then, there exists a crepant birational map $(\mathbb{P}^2,C)\dashrightarrow (X,B)$, where $(X,B)$ is log canonical.
\end{lemma}

\begin{proof}

We will start by describing some maps from $\mathbb{P}^2$ to each of these surfaces.

    Let $C$ be a nodal cubic in $\mathbb{P}^2$ with a flex $P$ and let $L$ be a line tangent to $C$ at $P$. 
    If we blow up $3,4,5,6,7$ or $8$ times  over $P$ along the strict transform of $C$, we end up with the minimal resolutions of $X(A_1+A_2)$, $X(A_4)$, $X(D_5)$, $X(E_6)$, $X(E_7)$ or $X(E_8)$, respectively.

    Let $C$ be a nodal cubic in $\mathbb{P}^2$ with a flex $P$, $L_1$ be the line tangent to $C$ at $P$ and $L_2$ another line passing through $P$, tangent to $C$ at some other point $Q$.
    If we blow up $5$ times  over $P$ along the strict transform of $C$, and we blow up over $Q$ two times along the strict transform of $C$, we end up with the minimal resolution of $X(D_6+A_1)$.
    If we blow up once more over $Q$ at the strict transform of $C$, we end up with the minimal resolution of $X(D_8)$. 
     
     Let $C$ be a nodal cubic in $\mathbb{P}^2$ with a flex $P$, $L_1$ be the line tangent to $C$ at $P$ and $L_2$ another line passing through $P$, tangent to $C$ at some other point $Q$.
     If we blow up $6$ times  over $P$ and $2$ times over $Q$ along the strict transform of $C$, we end up with the minimal resolution of $X(E_7+A_1)$.
 
    Let $C$ be a nodal cubic in $\mathbb{P}^2$ with two flexes $P_1$ and $P_2$, and $L_1$ and $L_2$ be the lines tangent to $C$ at $P_1$ and $P_2$ respectively.
    If we blow up $3$ times over $P_1$ and $3$ times over $P_2$ along the strict transform of $C$, we end up with the minimal resolution of $X(A_2+A_2+A_2)$.
    If we furthermore blow up over $P_1$ along the strict transform of $C$ once more, we end up with the minimal resolution of $X(A_5+A_2)$.
    If we furthermore blow up over $P_1$ along the strict transform of $C$ once more, we end up with the minimal resolution of $X(E_6+A_2)$.

    Let $C$ be a nodal cubic in $\mathbb{P}^2$ with a flex $P$ and node $N$, let $L_1$ be the line tangent to $C$ at $P$, let $L_2$ be a line passing through $P$ tangent to $C$ at $Q$ and let $L_3$ be the line passing through $P$ and $N$.
    If we blow up $3$ times over $P$ and $2$ times over $Q$ along the strict transform of $C$. And we blow up twice over $N$ along the strict transform of $L_3$, then we have the minimal resolution of $X(A_1+A_1+A_1+D_4)$ .
    If we blow up once more over $Q$ along the strict transform of $C$, then we have the minimal resolution of $X(A_1+A_1+D_6)$.

    Let $C$ be a nodal cubic in $\mathbb{P}^2$ with a flex $P$, let $L_1$ be the line tangent to $C$ at $P$, let $L_2$ be a line passing through $P$ and tangent to $C$ at $Q_1$, and let $L_3$ be a line passing through $Q_1$ and tangent to $C$ at $Q_2$.
    If we blow up $3$ times over $P$, $2$ times over $Q_1$ and $2$ times over $Q_2$ along the strict transform of $C$, then we have the minimal resolution of $X(A_1+A_3+A_3)$.
    If we blow up once more over $Q_2$ along the strict transform of $C$, then we have the minimal resolution of $X(D_5+A_3)$.

    In any of these previous cases let $X$ be the Gorenstein del Pezzo surface of rank $1$. If we take the pair $(X,\hat{C})$, where $\hat{C}$ is the strict transform of $C$ via the maps described. The described maps from $(\mathbb{P}^2,C)$ to $(X,\hat{C})$ are crepant birational.
    
    Let $C$ be a nodal cubic in $\mathbb{P}^2$, with node $N$ and two flexes $P_1$ and $P_2$, and $L_1$ and $L_2$ be the lines tangent to $C$ at $P_1$ and $P_2$ respectively.
    If we blow up $2$ times over $N$, $3$ times over $P_1$ and $3$ times over $P_2$ along $C$, then we have the minimal resolution of $X(A_2+A_2+A_2+A_2)$.

    In this final case, let $B$ be the image on $X$ of the $(-1)$-curve on the minimal resolution that maps to $N$ in $\mathbb{P}^2$. Then the described map from $(\mathbb{P}^2,C)$ to $(X,B)$ is crepant birational.

    As $(\mathbb{P}^2,C)$ is crepant birational to $(\mathbb{P}^2,\Sigma^2)$, via two standard Cremona transformations, we obtain the desired result
\end{proof}

\begin{lemma}\label{lem:Gor-del-Pezzo-conic}
Let $X$ be a Gorenstein del Pezzo surface of Picard rank one, whose singularity type is $X(A_4+A_4)$ and let $(\mathbb{P}^2,Q+L)$ be the projective plane with a conic and a transversal line. Then, there exists a crepant birational map $(\mathbb{P}^2,Q+L)\dashrightarrow (X,B)$, where $(X,B)$ is log canonical.
\end{lemma}

\begin{proof}
    
    Let $Q$ be a conic with secant line $L_1$. Write $P_1$ and $P_2$ for the points of intersection between $Q$ and $L_1.$ Let $L_2$ be a line tangent to $Q$ at $P_2,$ and let $L_3$ be a line tangent to $Q$ at a point $P_3\in Q \setminus L_1.$ Write $P_4$ for the point of intersection between $L_3$ and $L_1.$ Blow up once at $P_1$ and once at $P_4.$ Blow up three times above $P_2$ along $L_2$ and three times above $P_3$ along $Q.$ We obtain the minimal resolution of $X(A_4+A_4).$

    As $(\mathbb{P}^2,Q+L)$ is crepant birational to $(\mathbb{P}^2,\Sigma^2)$, via one standard Cremona transformation, we obtain the desired result.
\end{proof}

\begin{lemma}\label{lem:Gor-del-Pezzo-lines}
Let $X$ be a Gorenstein del Pezzo surface of Picard rank one, whose singularity type is $\{A_8,A_7+A_1,A_1+A_2+A_5,A_1+A_1+A_3,A_1+A_5,A_7 \}$ and let $(\mathbb{P}^2,L_1+L_2+L_3)$ be the projective plane with a the three coordinate axes. Then, there exists a crepant birational map $(\mathbb{P}^2,L_1+L_2+L_3)\dashrightarrow (X,B)$, where $(X,B)$ is log canonical.
\end{lemma}

\begin{proof}
    Let $L_1$, $L_2$, $L_3$ be three lines in $\mathbb{P}^2$ where $P_1$ is the intersection of $L_1$ and $L_3$ and $P_2$ is the intersection of $L_2$ and $L_3$.
    Blow up $3$ times over $P_1$ at the strict transform of $L_1$ and $3$ times over $P_2$ at the strict transform of $L_1$. Next, blow up one interior point of each of the two exceptional $(-1)$-curves. We obtain the minimal resolution of $X(A_8).$ 

    Let $L_1$, $L_2$, $L_3$ be three lines in $\mathbb{P}^2$. Write $P_1$ for the point of intersection between $L_1$ and $L_2,$ $P_2$ for the the point of intersection between $L_2$ and $L_3,$ and $Q$ for a point in $L_1\setminus(L_2\cup L_3).$ Blow up twice over $Q$ along $L_1$ and three times over $P_2$ along $L_3.$ Blow up twice over $P_1$ along $L_2.$ Finally, blow up once at an interior point of the exceptional $(-1)$-curve over $P_1.$ We obtain the minimal resolution of $X(A_7+A_1).$

    Let $L_1$, $L_2$, $L_3$ be three lines in $\mathbb{P}^2$ where $P_1$ is the intersection of $L_1$ and $L_3$. Let $Q_1$ and $Q_2$ be points in $L_1$ and $L_2$ respectively, but in no other $L_i$.
    If we blow up $3$ times over $P_1$ along the strict transform of $L_3$, we blow up $2$ times over $Q_1$ along the strict transform of $L_1$ and we blow up $3$ times over $Q_2$ along the strict transform of $L_2$, then we have the minimal resolution of $X(A_1+A_2+A_5)$

    Let $L_1$, $L_2$, $L_3$ be three lines in $\mathbb{P}^2$ where $P_1$ is the intersection of $L_1$ and $L_3$ and $P_2$ is the intersection of $L_2$ and $L_3$.
    If we blow up $3$ times over $P_1$ along the strict transform of $L_1$ and we blow up $2$ times over $P_2$ along the strict transform of $L_3$, then we have the minimal resolution of $X(A_1+A_1+A_3)$.
    
    If furthermore we blow up once over $P_1$ at a point in the $(-1)$-curve, not in any $(-2)$-curve, then we have the minimal resolution of $X(A_1+A_5)$.
    
    If furthermore we blow up once more over $P_2$ at a point in the $(-1)$-curve, not in any $(-2)$-curve, then we have the minimal resolution of $X(A_7)$.

\end{proof}

\begin{proof}[Proof of Theorem~\ref{thm:bir-comp-Del-Pezzo-rank-1}]
    By \cite[Theorem 1.2]{Ye02}, the surfaces in Lemmas~\ref{lem:Gor-del-Pezzo-cubic},~\ref{lem:Gor-del-Pezzo-conic}, and \ref{lem:Gor-del-Pezzo-lines}, are all the Gorenstein del Pezzo surface of rank $1$, other than those listed in the statement of the theorem. Therefore, all the required surfaces admit a $1$-complement $B$ with $c_{\rm bir}(X,B)=0$. By Lemma~\ref{lem:no-one-complement} and \cite[Example 3.13]{Mor22}, the surfaces listed in the statement of the Theorem do not admit a $1$-complement with coregularity $0$, those a fortiori they cannot admit a $1$-complement of birational complexity $0$.
\end{proof}

\section{Torus exceptional degree}
In this section, we prove several statements about the torus exceptional degree of log Calabi--Yau pairs.
First, we are concerned with the general picture, and then we focus on log Calabi--Yau surfaces.

\subsection{Log Calabi--Yau pairs}
In this subsection, we prove the positivity of the torus exceptional degree
and the characterization
of log Calabi--Yau pairs 
with torus exceptional degree zero.

\begin{proof}[Proof of Theorem~\ref{introthm:t-equal-zero}]
Let $(X,B)$ be a log Calabi--Yau pair of dimension $n$ and index one. 
By definition, we have that 
$\mathcal{T}_t(X,B)\geq \mathcal{T}(X,B)\geq 0$. This proves (1).
If $\mathcal{T}_t(X,B)$ is finite, then the log Calabi--Yau pair admits a crepant birational map
$(\pp^n,\Sigma^n)\dashrightarrow (X,B)$ so its birational complexity is zero.
On the other hand, assume that $(X,B)$ has birational complexity zero.
Then, by the proof of Theorem~\ref{introthm:covering-bcomp0},
there is a finite set $\{\phi_i\}_{i\in I}$
of crepant birational maps
$\phi_i\colon (\pp^n,\Sigma^n)\dashrightarrow (X,B)$ that divisorially cover $(X,B)$.
Thus, we get 
\[
\mathcal{T}_t(X,B) \leq \sum_{i\in I} \mathcal{T}(\phi_i), 
\]
and the right-hand side is finite. 
This proves (2).

We turn to prove (3).
If $(X,B)$ is divisorially covered by almost tori, then there are birational maps
\[
\varphi_i \colon \mathbb{P}^n 
\dashrightarrow X, 
\]
that induce embeddings 
$\varphi_i \colon \mathbb{G}_m^n\setminus Z_i 
\hookrightarrow X\setminus B$ where $Z_i$ is a closed subvariety of $\mathbb{G}_m^n$ of codimension at least $2$.
Let $p_i\colon Y_i\rightarrow \mathbb{P}^n$ 
and $q_i\colon Y_i\rightarrow X$ give a resolution of the birational map $\phi_i$.
Let $\Delta_i:={p_i}_*q_i^*B$.
Since $\mathbb{G}_m^n \setminus Z_i \hookrightarrow X\setminus B$ is an embedding,
we conclude that $\Delta_i$ is supported on 
$\mathbb{P}^n \setminus \mathbb{G}_m^n$.
This means that $\Delta_i$ is supported on the hyperplane coordinates.
As $(\pp^n,\Delta_i)$ is a sub log Calabi--Yau pair, we conclude that $\Delta_i=\Sigma^n$ for each $i$.
Thus, we have a finite set of crepant birational maps
\[
\phi_i\colon (\pp^n,\Sigma^n) \dashrightarrow (X,B),
\]
that divisorially cover $(X,B)$.
Furthermore, for each $i$, we have 
$\mathcal{T}(\phi_i)=0$.
Indeed, the birational maps
$\phi_i\colon \pp^n\dashrightarrow X$ are small in the algebraic torus.
We conclude that $\mathcal{T}(X,B)=\mathcal{T}_t(X,B)=0$.

Now, assume that $\mathcal{T}(X,B)=0$.
This means that there exists a crepant birational map
\[
\phi \colon (\pp^n,\Sigma^n) \dashrightarrow (X,B),
\]
such that $\mathcal{T}(\phi)=0$.
In other words, no exceptional divisor of $\phi$ intersects the algebraic torus $\mathbb{G}_m^n$.
Let 
\[
U_1:=\phi(\pp^n\setminus (\Sigma^n \cup {\rm Ex}(\phi))).
\]
Then $U_1$ is an almost torus that admits an embedding $U_1\hookrightarrow X\setminus B$.
Let $E_2,\dots,E_k$ be the divisorial
prime components
of $(X\setminus B)\setminus U_1$.

For each $i\in \{2,\dots,k\}$, we show that there exists a crepant birational map
$\phi_i\colon (\pp^n,\Sigma^n)\dashrightarrow (X,B)$ such that the center of $E_i$ on $\pp^n$ is a divisor intersecting $\mathbb{G}_m^n$
and $\phi_i$ is of cluster type, i.e., induces a small birational map on $\mathbb{G}_m^n$.
This would imply that the open sets
\[
U_i:=\phi_i(\pp^n\setminus (\Sigma^n \cup {\rm Ex}(\phi_i)))
\]
are almost tori that cover $X\setminus B$ divisorially. 

It suffices to show the statement for $E:=E_2$.
The divisor $E$ is exceptional over the domain $\pp^n$ of $\phi$.
Then, we can apply Lemma~\ref{lem:special-model-complexity-1}
to the pair $(\pp^n,\Sigma^n)$ and the valuation $E$.
We obtain a crepant birational map 
$\phi\colon (Y,B_Y)\dashrightarrow (X,B)$ satisfying the following conditions:
\begin{enumerate}
    \item we have $c(Y,B_Y)=1$,
    \item the center of $E$ on $Y$ is a prime divisor $E_Y$, 
    \item the variety $Y$ is $\qq$-factorial and admits a fibration $p\colon Y\rightarrow Z$ to a variety of dimension $n-1$,
    \item there is a prime component $F_Y$ of $p^{-1}(p(E_Y))$ that dominates $E_Y$, is not contained in $B_Y$, and it is different from $E_Y$, and 
    \item the pair $(Y,B_Y+\epsilon F_Y)$ is dlt for $\epsilon>0$ small enough.
\end{enumerate}
Furthermore, we know that the crepant birational map $(Y,B_Y)\dashrightarrow (\pp^n,\Sigma^n)$ only extracts $E_Y$ and 
log canonical places of $(\pp^n,\Sigma^n)$.
By condition (5),
we can run a $(K_Y+B_Y+\epsilon F_Y)$-MMP over $Z$. This MMP terminates with a toric log Calabi--Yau pair $(T,B_T)$ due to conditions (1) and (3).
Let $E_T$ be the center of $E$ on $T$.
Note that $(T,B_T) \dashrightarrow (\pp^n,\Sigma^n)$ is a crepant birational map that only extracts $E_T$ and log canonical places of $(\pp^n,\Sigma^n)$.
Since the center of $E_T$ on $X$ is a divisor, we conclude that the crepant birational map
$f\colon (T,B_T)\dashrightarrow (X,B)$ is small on $T\setminus B_T$.
Let $g\colon (T,B_T)\dashrightarrow (\pp^n,\Sigma^n)$ be any crepant toric birational map
to $(\pp^n,\Sigma^n)$.
Toric birational maps are isomorphisms on the algebraic torus $\mathbb{G}_m^n$.
We conclude that the crepant birational map
\[
\phi_2:= g^{-1}\circ f \colon 
(\pp^n,\Sigma^n)\dashrightarrow (X,B)
\]
is of cluster type 
and the center of $E$ on its domain $\pp^n$ is a divisor that intersects $\mathbb{G}_m^n$ non-trivially.
Hence, the almost tori $U_2:=\phi_2(\pp^n\setminus(\Sigma^n\cup {\rm Ex}(\phi_2)))$ intersects the divisor $E_2$ non-trivially. This finishes the proof of (3).

Now, we turn to prove (4).
Assume that $X$ is Fano and $\mathcal{T}(X,B)=0$.
By (3), we know that $(X,B)$ is divisorially covered by finitely many almost tori.
It suffices to show that, in this case, such almost tori are indeed algebraic tori.
Let $\{\phi_i\}_{i\in I}$ be a finite set of crepant birational maps
$\phi_i\colon (\pp^n,\Sigma^n)\dashrightarrow (X,B)$ of cluster type
such that $V_i:=\phi_i(\mathbb{G}_m^n\setminus {\rm Ex}(\phi_i))$ divisorially cover $X\setminus B$.   
Let $P_1,\dots,P_k$ be the prime exceptional divisors of $\phi_i^{-1}$ that are not contained in $B$.
As $X$ is Fano, the variety $X\setminus B$ is affine, and so 
$U_i:=X\setminus \supp(B+m_1P_1+\dots+m_kP_k)$ is affine, where $m_i$ is the Cartier index of $P_i$.
The birational map $\phi_i^{-1}$ induces a small birational map between the affine variety $U_i$ and the affine variety $\mathbb{G}_m^n$.
Note that $V_i$ and $U_i$ are small birational.
So, $\phi_i^{-1}$ induces an isomorphism between $U_i$ and $\mathbb{G}_m^n$.
We conclude that $X\setminus B$ is divisorially covered by finitely many algebraic tori.
\end{proof}

\subsection{Log Calabi--Yau surfaces}
In this subsection, we study the torus exceptional degree of log Calabi--Yau surfaces. 
In order to prove Theorem~\ref{introthm:bound-ted-gor-dP}, we first establish some lemma regarding particular birational maps in $\mathbb{P}^2$.

\begin{lemma}\label{lem:ted-cubic-P2}
    Let $C$ be a nodal cubic in $\mathbb{P}^2$ with node $N$ and two other points $P_1$ and $P_2$. Let $L_i$ be some lines, each one passing through at least one of the points $P_1$, $P_2$ or $N$. Let $Q_j$ be some conics, each one passing through at least two of the points $P_1$, $P_2$ or $N$. There exists a crepant birational map $\varphi:(\mathbb{P}^2,\Sigma^2)\dashrightarrow (\mathbb{P}^2,C)$, with $\tau (\varphi) \leq 3$, such that the strict transform of each $L_i$ is empty or a curve of degree at most $2$, and the strict transforms of each $Q_j$ is empty or a curve of degree at most $4$.
\end{lemma}

\begin{proof}
    Let $\varphi_1$ be the standard Cremona transformation with centers at $N,P_1, P_2$. Let $L_i'$, $Q_j$ and $C'$ be the strict transforms of $L_i$, $Q_j$ and $C$ respectively.
    By the choice of centers, $C'$ is a conic, $L_i'$ are lines or points and $Q_j'$ have degrees $1$ or $2$.
    Let $P$ the strict transform of the line $P_1P_2$, and $N_1, N_2$ be the images of the points over $N$ lying on the strict transform of $C$. 
    Let $\varphi_2$ be the standard Cremona transformation with centers $N_1,N_2,P$. Let $C''$, $L_i''$ and $Q_j''$ be the strict transforms of $C'$, $L_i'$ and $Q_j'$. By the choice of centers $C''$ is a line and the degrees of $L_i''$ and $Q_j''$ are at most $2$ and $4$ respectively.
    Let $E_{N_1}$ and $E_{N_2}$ be the images of the blow ups at $N_1$ and $N_2$.

    The birational map $\varphi_1^{-1} \circ \varphi_2^{-1}:(\mathbb{P}^2,C''+E_{N_1}+E_{N_2})\dashrightarrow (\mathbb{P}^2,C)$ is crepant birational, with $\tau(\varphi_1^{-1} \circ \varphi_2^{-1})=3$, satisfying the desired properties.
\end{proof}

\begin{lemma}\label{lem:ted-conic-P2}
        Let $Q$ be a conic in $\mathbb{P}^2$ and $L$ a line transversal to $Q$. Let $C_i$ be some curves in $\mathbb{P}^2$. There exists a crepant birational map $\phi:(\mathbb{P}^2,\Sigma^2)\dashrightarrow (\mathbb{P}^2,Q+L)$, with $\tau(\phi)= 1$, such that the strict transform of each $C_i$ is empty or a curve of degree at most twice the degree of $C_i$.
\end{lemma}

\begin{proof}
    Let $P_1,P_2$ be the intersection points of $Q$ and $L$, and let $P_3$ be another point in $Q$.
    Let $\varphi$ be the standard Cremona transformation center at $P_1,P_2,P_3$. 
    Then $\varphi^{-1}$ satisfies the desired properties.
\end{proof}

\begin{proof}[Proof of Theorem~\ref{introthm:bound-ted-gor-dP}]

By Lemma~\ref{lem:Gor-del-Pezzo-cubic}, \ref{lem:Gor-del-Pezzo-conic} and \ref{lem:Gor-del-Pezzo-lines} any Gorenstein del Pezzo surface of Picard rank one, that has a 1-complement of birational complexity $0$, admits a crepant birational map from $(\mathbb{P}^2, C)$,  $(\mathbb{P}^2, Q+L)$ or $(\mathbb{P}^2, \Sigma^2)$ of the following form:
\[
\xymatrix{
& \ar[ld]_{f} (Y,B'_Y) \ar[rd]^{g} & \\
(X,B')  & &\ar@{-->}[ll]_{\phi}(\mathbb{P}^2,B'_{\mathbb{P}^2}).
}
\]
Here, $f$ is a resolution of singularities and $g$ is a sequence of blow ups. Let the log pull-back of $(X,B)$ be $(Y,B_Y)$ and the pushforward $g_*B_Y$ be $B_{\mathbb{P}^2}$. Since $(X,B)$ is log Calabi--Yau of index one and the singularities of $X$ are canonical, the pair $(Y,B_Y)$ is also log Calabi--Yau of index one and so is $(\mathbb{P
}^2,B_{\mathbb{P}^2})$.

We separate in cases, depending on the isomorphism class of $X$.

\begin{enumerate}
    \item The surface $X$ satisfies the hypothesis of  Lemma~\ref{lem:Gor-del-Pezzo-cubic}:

    By the proof of Lemma~\ref{lem:Gor-del-Pezzo-cubic}, the exceptional divisor of $\phi$ consists of one, two or three lines, all of which pass through one of two points $R_1$, $R_2$ in $C$. Furthermore, by the way the blow ups are performed in $g$, there are $(-1)$-curves in $Y$ that map to the points $R_1$ and $R_2$. These $(-1)$-curves have intersection $1$ with $B_Y$, thus $B_{\mathbb{P}^2}$ must pass through $R_1$ and $R_2$. Hence, if $B_{\mathbb{P}^2}$ is a nodal cubic, or a conic and a line transversal to it, then we satisfy the hypothesis of Lemmas~\ref{lem:ted-conic-P2} or \ref{lem:ted-cubic-P2}.
    
    As the pair $(\mathbb{P}^2,B_{\mathbb{P}^2})$  is of coregularity $0$, there are three possibilities for $B_{\mathbb{P}^2}$.

    \begin{enumerate}
        \item If $B_{\mathbb{P}^2}$ are $3$ lines, then by the crepant birational map $\phi$, we have that $\mathcal{T}(X,B)\leq 3$, since the exceptional locus in any of the cases of proof of Lemma~\ref{lem:Gor-del-Pezzo-cubic} consists of at most three lines.

        \item If $B_{\mathbb{P}^2}$ is a conic and a line, then let $\varphi_2:(\mathbb{P}^2,\Sigma^2)\dashrightarrow (\mathbb{P}^2, B_{\mathbb{P}^2})$  be the crepant birational map from Lemma~\ref{lem:ted-conic-P2}. Thus by the crepant birational map $\phi \circ \varphi_2:(\mathbb{P}^2,\Sigma^2)\dashrightarrow (X,B)$, we have that $\mathcal{T}(X,B)\leq 7$, since the exceptional locus of $\phi$ in any of the cases of Lemma~\ref{lem:Gor-del-Pezzo-cubic} consists of at most three lines.

        \item If $B_{\mathbb{P}^2}$ is a nodal cubic, then let $\varphi_3:(\mathbb{P}^2,\Sigma^2)\dashrightarrow (\mathbb{P}^2, B_{\mathbb{P}^2})$  be the crepant birational map from Lemma~\ref{lem:ted-cubic-P2}. By the crepant birational map $\phi \circ \varphi_3:(\mathbb{P}^2,\Sigma^2)\dashrightarrow (X,B)$, we have that $\mathcal{T}(X,B)\leq 9$, since the exceptional locus of $\phi$ in any of the cases of Lemma~\ref{lem:Gor-del-Pezzo-cubic} consists of at most three lines.
    \end{enumerate}

    \item The surface $X$ satisfies the hypothesis of Lemma~\ref{lem:Gor-del-Pezzo-conic}

    By the proof of Lemma~\ref{lem:Gor-del-Pezzo-conic}, the exceptional divisor of $\phi$ consists of a line and a conic, whose intersections are $R_1$ and $R_2$. Furthermore, by the way the blow ups are performed in $g$, there are $(-1)$-curves in $Y$ that map to the points $R_1$ and $R_2$. These $(-1)$-curves have intersection $1$ with $B_Y$, thus $B_{\mathbb{P}^2}$ must pass through $R_1$ and $R_2$. Hence, if $B_{\mathbb{P}^2}$ is a nodal cubic, or a conic and a line transversal to it, then we satisfy the hypothesis of Lemmas~\ref{lem:ted-conic-P2} or \ref{lem:ted-cubic-P2}.
    
    As the pair $(\mathbb{P}^2,B_{\mathbb{P}^2})$  is of coregularity $0$, there are three possibilities for $B_{\mathbb{P}^2}$.

    \begin{enumerate}
        \item If $B_{\mathbb{P}^2}$ are $3$ lines, then by the crepant birational map $\phi$, we have that $\mathcal{T}(X,B)\leq 3$, since the exceptional locus of $\phi$ is a line and a conic.

        \item If $B_{\mathbb{P}^2}$ is a conic and a line, then let $\varphi_2:(\mathbb{P}^2,\Sigma^2)\dashrightarrow (\mathbb{P}^2, B_{\mathbb{P}^2})$  be the crepant birational map from Lemma~\ref{lem:ted-conic-P2}. Thus by the crepant birational map $\phi \circ \varphi_2:(\mathbb{P}^2,\Sigma^2)\dashrightarrow (X,B)$, we have that $\mathcal{T}(X,B)\leq 7$, since the exceptional locus of $\phi$ consists of a line and a conic.

        \item If $B_{\mathbb{P}^2}$ is a nodal cubic, then let $\varphi_3:(\mathbb{P}^2,\Sigma^2)\dashrightarrow (\mathbb{P}^2, B_{\mathbb{P}^2})$  be the crepant birational map from Lemma~\ref{lem:ted-cubic-P2}. By the crepant birational map $\phi \circ \varphi_3:(\mathbb{P}^2,\Sigma^2)\dashrightarrow (X,B)$, we have that $\mathcal{T}(X,B)\leq 9$, since the exceptional locus of $\phi$ consists of a line and a conic.
    \end{enumerate}

    \item The surface $X$ satisfies the hypothesis of Lemma~\ref{lem:Gor-del-Pezzo-lines}

    By the proof of Lemma~\ref{lem:Gor-del-Pezzo-lines}, the exceptional divisor of $\phi$ consists of three lines. Furthermore by the way the blow ups are performed in $g$, there are $(-1)$-curves in $Y$ that map to two of the  intersection points of these lines, let us call them $R_1$ and $R_2$. These $(-1)$-curves have intersection $1$ with $B_Y$, thus $B_{\mathbb{P}^2}$ must pass through $R_1$ and $R_2$. Hence, if $B_{\mathbb{P}^2}$ is a nodal cubic, or a conic and a line transversal to it, then we satisfy the hypothesis of Lemmas~\ref{lem:ted-conic-P2} or \ref{lem:ted-cubic-P2}.
    
    As the pair $(\mathbb{P}^2,B_{\mathbb{P}^2})$  is of coregularity $0$, there are three possibilities for $B_{\mathbb{P}^2}$.

    \begin{enumerate}
        \item If $B_{\mathbb{P}^2}$ are $3$ lines, then by the crepant birational map $\phi$, we have that $\mathcal{T}(X,B)\leq 3$, since the exceptional locus of $\phi$ consists of three lines.

        \item If $B_{\mathbb{P}^2}$ is a conic and a line, then let $\varphi_2:(\mathbb{P}^2,\Sigma^2)\dashrightarrow (\mathbb{P}^2, B_{\mathbb{P}^2})$  be the crepant birational map from Lemma~\ref{lem:ted-conic-P2}. Thus by the crepant birational map $\phi \circ \varphi_2:(\mathbb{P}^2,\Sigma^2)\dashrightarrow (X,B)$, we have that $\mathcal{T}(X,B)\leq 7$, since the exceptional locus of $\phi$ consists of three lines.

        \item If $B_{\mathbb{P}^2}$ is a nodal cubic, then let $\varphi_3:(\mathbb{P}^2,\Sigma^2)\dashrightarrow (\mathbb{P}^2, B_{\mathbb{P}^2})$  be the crepant birational map from Lemma~\ref{lem:ted-cubic-P2}. By the crepant birational map $\phi \circ \varphi_3:(\mathbb{P}^2,\Sigma^2)\dashrightarrow (X,B)$, we have that $\mathcal{T}(X,B)\leq 9$, since the exceptional locus of $\phi$ consists of three lines.
    \end{enumerate}
    
\end{enumerate}

\end{proof}

Now, we turn to prove Theorem~\ref{introthm:bounded-t-surf}.
First, we state a lemma which is a consequence of~\cite[Lemma 2.12]{GLM23}.

\begin{lemma}\label{lem:bdry-in-smth-locus}
    Let $(X,B)$ be a dlt log Calabi--Yau surface of index one. If $f\colon X\rightarrow Y$ is a composite of divisorial contractions, then $B_Y=f_*B$ is contained in the smooth locus of $Y$.
\end{lemma}

\begin{proof}[Proof of Theorem~\ref{introthm:bounded-t-surf}]
Let $(X,B)$ be a log Calabi--Yau surface of index one. Let $f\colon(X',B')\rightarrow (X,B)$ be a dlt modification. Then $X'$ is canonical and $B'$ is contained in the smooth locus of $X'.$ We may run a $K_{X'}$-MMP to obtain $$X'\xrightarrow{g}X''\xrightarrow{\pi}Z,$$ with $g$ a composite of divisorial contractions and $\pi$ a Mori fiber space. We note that $X''$ is canonical. By Lemma~\ref{lem:bdry-in-smth-locus}, $B''=g_*B'$ is contained in the smooth locus of $X''.$ Note that for any crepant birational map $\phi:(\pp^2,\Sigma^2)\dashrightarrow (X,B)$ we have 
\[
\mathcal{T}(\phi)\leq \mathcal{T}(g\circ f^{-1}\circ \phi)
\]
since $g$ is a morphism and $f$ only extracts log canonical places of $(X,B)$. It follows that we have an inequality 
\[
\mathcal{T}(X,B)\leq \mathcal{T}(X'',B''),
\]
and thus it suffices to consider index one log Calabi--Yau surfaces $(X,B)$ satisfying the following: 
\begin{enumerate}
    \item $X$ is canonical,
    \item $X$ is equipped with a Mori fiber space $\pi\colon X\rightarrow Z$, and 
    \item $B \subset X_{reg}$ and $B \neq \emptyset.$
\end{enumerate} \par
If $Z$ is a point, then $X$ is a rank one Gorenstein del Pezzo surface. In this case, the statement follows from Theorem~\ref{introthm:bound-ted-gor-dP}.

From now on, we assume that $Z\simeq \pp^1$ and $\pi$ has irreducible fibers. We have $B\cdot F=2$ for every fiber $F$ of $\pi,$ and since $B$ is contained in the smooth locus of $X$ we have that $B\cdot F_{red}$ is an integer for each fiber $F$. Since $\pi$ can only have multiple fibers it follows that either $F$ is contained in the smooth locus of $X$ or that $B\cdot F_{red}=1.$ In particular, there are no singular fibers that contain a node of $B.$ Thus, after possibly performing an elementary transformation centered at some node of $B$, we may assume that $B$ has at least two components and that at least one of them is a fiber.
If $B$ has at least three components, then at least one component $B_0$ satisfies $B_0\cdot F=1$ for each fiber $F.$ But this means there can be no multiple fibers, and so $X$ must be smooth. In particular, $X$ is a Hirzebruch surface. After performing at most one additional elementary transformation centered at a node of $B$, we may assume that $B$ contains two fibers of $\pi$ and hence is the toric boundary for some toric structure on $X$.
In the previous process, we only extracted log canonical places of $(X,B)$, so $\mathcal{T}(X,B)=0$ in this case.

To conclude, we consider the case in which $B$ has only two components. Both components are isomorphic to $\pp^1$, one component $B_1$ is a fiber of $\pi$, and the other component $B_2$ is a $2$-section. Each singular fiber meets $B_2$ in precisely one point, hence it corresponds to a ramification point of the morphism $B_2\to \pp^1$ induced by $\pi.$ By Riemann-Hurwitz, there can be at most two singular fibers. We resolve the singularities of $X$ and run a relative MMP over $\pp^1.$ We obtain a crepant birational map
\[
\xymatrix{
(X,B)\ar[rd]_-{\pi}\ar@{-->}[rr]^-{g} & 
 & (X',B')\ar[dl]^-{\pi'}\\
& \pp^1. & 
}
\]
over $\pp^1$, where $X'$ is a Hirzebruch surface and $B'$ is a $1$-complement consisting of a fiber $B'_1$ of $\pi'$ and a $2$-section $B'_2$. Write $F_1$ and $F_2$ for the two fibers of $\pi'$ meeting $B'_2$ at ramification points of the induced morphism $B'_2\to \pp^1.$ 
Observe that $F_1$ and $F_2$ are canonical places of $(X,B)$ extracted on $X'$.

We note that $X'$ is either $\pp^1\times \pp^1$ or the blow up $\mathbb{F}_1$ of $\pp^2$ at a point. Indeed, let $C\subset X'$ be a section with $C^2=-n$, $n\geq 0$. Since $C$ is distinct from both $B'_1$ and $B'_2$ and since $B'_1\cdot C=1,$ we have $$2-n=-K_{X'}\cdot C=B'\cdot C\geq 1,$$ hence $n\leq 1.$ If $X'\cong \pp^1\times \pp^1,$ then we may perform an elementary transformation centered at a node of $B'$ to obtain a 1-complement of $\mathbb{F}_1$. with two components. Thus, we may assume that $X'\cong \mathbb{F}_1$. Note that the negative section on $X'$ does not meet the 2-section $B'_2.$ Blowing down the negative section, we obtain a complement $B''$ on $\pp^2$ consisting of a line $B''_1$ and a conic $B''_2$ intersecting transversely, and the images of $F_1$ and $F_2$ are tangent lines to $B''_2$. We perform a standard Cremona transformation centered at the points of intersection between $B_2''$ and $F_1,$ $B''_2$ and $F_2$, and one of the intersection points of $B''_1$ and $B''_2$. We obtain a crepant toric model for $(X,B)$ whose exceptional locus consists of at most four lines, namely two of the exceptional divisors of the Cremona together with some subset of the strict transforms of $F_1$ and $F_2.$ It follows in this case that $\mathcal{T}(X,B)\leq 4.$
\end{proof}

\begin{proof}[Proof of Theorem~\ref{introthm:bcomp-zero-surf-aff-bounded}]
    By Theorem~\ref{introthm:bounded-t-surf} any projective surface in $\mathcal{F}_2$ has a log pair structure $(X,B)$, with $\mathcal{T}(X,B)\leq 9$. This implies that there exists a birational map $\varphi: \mathbb{G}_m^2 \dashrightarrow X\setminus B$, with ${\rm Ex}(\varphi)$ of degree at most $9$.
    Thus there exists an open subset of $X$ isomorphic to $\mathbb{A}^2\setminus D$, where $D$ is a curve of degree at most $11$. These open subsets are fibers of the morphism:
    \[
    {\rm Spec}(\kk[x,y,a_{0,0}, a_{0,1}\ldots,a_{11,11}]_{\Sigma_{i,j} a_{i,j}x^iy^{j}})\rightarrow {\rm Spec}(\kk[a_{0,0},a_{0,1}\ldots,a_{11,11}]).
    \]
    Thus $\mathcal{F}_2$ is affinely bounded.
\end{proof}

\section{Cluster type} 
In this section, we prove several statements about the cluster type
of a log Calabi--Yau pair. 
First, we characterize log Calabi--Yau toric pairs as those having cluster type one.
Secondly, we study the cluster type of Fano surfaces.

\subsection{Log Calabi--Yau pairs}
In this subsection, we prove the characterization of log Calabi--Yau toric pairs using the cluster type.

\begin{proof}[Proof of Theorem~\ref{theorem:cluster-type-1}]
Let $(X,B)$ be a log Calabi--Yau pair. 
By definition the cluster type $\mathcal{C}(X,B)$ of $(X,B)$ is either infinite
or a finite number at least one.

If $(X,B)$ is a toric log Calabi--Yau pair, then
$X\setminus B \simeq \mathbb{G}_m^n$ and so 
$\mathcal{C}(X,B)=1$.
Now, assume that $\mathcal{C}(X,B)=1$.
Then, there exists a crepant birational map
\[
\phi\colon (\pp^n,\Sigma^n)\dashrightarrow (X,B)
\]
that is of cluster type
and $U:=\phi(\pp^n\setminus (\Sigma^n\cup {\rm Ex}(\phi)))$ divisorially covers $X\setminus B$.
In other words, every exceptional divisor of $\phi^{-1}$ is contained in $B$.
Thus, there exists a crepant birational morphism 
$\psi \colon (T,B_T)\rightarrow (\pp^n,\Sigma^n)$, that only extracts log canonical places of $(\pp^n,\Sigma^n)$
such that the induced birational map
$\varphi \colon (T,B_T)\dashrightarrow (X,B)$ is a birational contraction.
Henceforth, we get a commutative diagram
\[
\xymatrix{
(T,B_T)\ar[d]^-{\psi}\ar@{-->}[rd]^-{\varphi} & \\
(\pp^n,\Sigma^n)\ar@{-->}[r]^-{\phi} & (X,B).
}
\]
Every log canonical place over a toric pair
is a toric valuation.
Hence, the log Calabi--Yau pair $(T,B_T)$
is a toric log Calabi--Yau pair.
We conclude that $(X,B)$ is toric, since $\varphi$ is a birational contraction ~\cite[Lemma 2.3.2]{BMSZ18} . 
\end{proof}

\subsection{Fano surfaces of cluster type}
In this subsection, we study the cluster type of Fano surfaces.
Our main result is that a Fano surface has cluster type either one, two, or infinite.

\begin{proposition}\label{prop:only-A-type}
Let $X$ be a surface of cluster type. Then $X$ has only $A$-type singularities.
\end{proposition}

\begin{proof}
Let $\phi\colon (\pp^2,\Sigma^2)\dashrightarrow (X,B)$ be a crepant birational map
of cluster type.
By Lemma~\ref{lem:1-comp-DE}, we know that $X$ has $A$-type singularities along $B$.
Thus, it suffices to show that $X$ only has $A$-type singularities along $X\setminus B$.
Let $B_1,\dots,B_k$ be the components of $B$.
Let $C_1,\dots,C_j$ be the divisors on $X$
exceptional over $\pp^2$ 
that are not contained in $B$.
Let $p_1\colon (T,B_T)\rightarrow (\pp^2,\Sigma^2)$ be a toric crepant birational morphism satisfying the following conditions:
\begin{itemize}
\item the center of each $B_i$ on $T$ is a divisor, and 
\item every divisor $C_i$ can be obtained from $T$ by a sequence of blow ups at points of $B_T$ that are not zero strata.
\end{itemize}
Let $t_1,\dots,t_s$ be the centers of the curves $C_i$ in $T$.
Let $p\colon X_0\rightarrow T$ be a projective birational morphism obtained by consecutively blowing up the points $t_1,\dots,t_s$ such that each $C_i$ is indeed a divisor on $X_0$. 
Let $(X_0,B_0)$ be the crepant pull-back of $(T,B_T)$ to $X_0$.
For each $i\in \{1,\dots,s\}$ the divisor
$p^{-1}(t_i)$ is a chain $E_{i,1}\cup \dots \cup E_{i,j_i}$ of rational curves such that only 
$E_{i,j_i}$ intersects $B_0$.
We have a crepant projective birational morphism
$q\colon (X_0,B_0)\rightarrow (X,B)$.
Observe that, as $\phi$ is of cluster type, 
the projective birational morphism
$q$ can only contract curves on 
$B_0\cup \bigcup_{i=1}^j\left(\bigcup_{k=1}^{j_i}E_{i,k} \right)$.
Let $x\in X\setminus B$ be a singular point
and let $F=q^{-1}(x)$.
If $F$ intersects $B_0$, then $x$ is contained in $B$ and hence $(X;x)$ is an $A$-type singularity.
If $F$ does not intersect $B_0$, then 
$F$ is contained in $\bigcup_{i=1}^j\left(\bigcup_{k=1}^{j_i}E_{i,k} \right)$. 
We conclude that $F$ is contained in a disjoint union of chains of rational curves.
As $F$ is connected, it must be a chain of rational curves itself so $(X;x)$ is an $A$-type singularity.
\end{proof}

\begin{lemma}\label{lem:config-curves-outside}
Let $X$ be a Fano surface and $(X,B)$ be a complement.
Let $\phi\colon (\pp^2,\Sigma^2)\dashrightarrow (X,B)$ be a crepant birational map of cluster type.
Let $C_1,\dots,C_k$ be the curves of $X$
that are not contained on $B$ and are exceptional over $\pp^2$.
Then for $i\neq j$ the curves $C_i$ and $C_j$ can only intersect along $B$.
\end{lemma}

\begin{proof}
Let $\phi\colon (\pp^2,\Sigma^2)\dashrightarrow (X,B)$ be a crepant birational map of cluster type.
Let $B_1,\dots,B_\ell$ be the components of $B$.
Let $p_1\colon (T,B_T)\rightarrow (\pp^2,\Sigma^2)$ be a toric crepant birational morphism satisfying the following conditions:
\begin{itemize}
\item the center on $T$ of each $B_i$ is a divisor,
\item every $C_i$ can be obtained from $T$ by a sequence of blow ups at points of $B_T$ that are not zero strata.
\end{itemize}
Let $t_1,\dots,t_s$ be the centers of the curves $C_i$'s in $T$. 
Let $p\colon X_0\rightarrow T$ be a projective birational morphism obtained by consecutively blowing up the points $t_i$'s such that each $C_i$ is a divisor on $X_0$.
We may choose $p$ minimal with this property.
For each $i\in \{1,\dots,s\}$ the divisor 
$p^{-1}(t_i)$ is a chain $E_i:=E_{i,1}\cup\dots\cup E_{i,\ell_i}$ of rational curves.
Note that only $E_{i,\ell_i}$ intersects $B_0$.
We have a crepant birational morphism
$q\colon (X_0,B_0)\rightarrow (X,B)$.
Since $\phi$ is of cluster type, the morphism $q$ only contracts curves in $B_0\bigcup_{i=1}^s\left(\bigcup_{k=i}^{\ell_i}E_{i,j} \right)$.
As we chose $p$ to be minimal, the curves $E_{i,\ell_i}$ is not contracted.
Since $B$ is ample the image of each $E_{i,j}$ on $X$ is either a point or a curve intersecting $B$.
Hence, all the curves $E_{i,1},\dots,E_{i,\ell_i-1}$ must be contracted. Otherwise, their images would not intersect $B$ in $X$, leading to a contradiction of the Fano condition.
As the chains $E_1,\dots,E_s$ are disjoint, we conclude that the curves $\phi(E_{1,j_1}),\dots, \phi(E_{s,\ell_s})$ can only intersect on $X$ along $B$.
By construction, we conclude that the curves
$\phi(E_{1,j_1}),\dots,\phi(E_{s,\ell_s})$ 
coincide with the curves $C_1,\dots,C_j$, up to reordering. This finishes the proof.
\end{proof}

\begin{proof}[Proof of Theorem~\ref{introthm:cluster-Fsurf}]
If $(X,B)$ is not a cluster type pair, then
$\mathcal{C}(X)=\infty$ and we are done.
Hence, from now on, we assume that $(X,B)$ is of cluster type and we aim to show that either
$\mathcal{C}(X,B)=1$ or $\mathcal{C}(X,B)=2$.

Let $X$ be a Fano surface and 
$(X,B)$ be a cluster type log Calabi--Yau pair.
Let $\phi \colon (\pp^2_{\phi},\Sigma^2)\dashrightarrow (X,B)$ be a crepant birational map of cluster type.
Let $C_1,\dots,C_k$ be the curves of $X$ that are not contained in $B$ and are exceptional over $\pp^2_{\phi}$.
If $k=0$, then $\phi_1(\mathbb{G}_m^2)$ divisorially covers $X\setminus B$
and $\mathcal{C}(X,B)=1$.
From now on, we assume that $k\geq 1$.
For each crepant birational map
$\varphi\colon (\pp^2_{\varphi},\Sigma^2)\dashrightarrow (X,B)$ of cluster type, we denote by 
$M(\varphi)$ the number of curves in $\{C_1,\dots,C_k\}$ that are not exceptional over $\pp^2_{\varphi}$.
If there exists such a $\varphi$ with $M(\varphi)=k$, then we are done. 
Indeed, in this case, the almost tori
$\phi(\mathbb{G}_m^2)$ and $\varphi(\mathbb{G}_m^2)$ divisorially cover $X\setminus B$ and hence $\mathcal{C}(X,B)=2$.

Let $\psi\colon (\pp^2_\psi,\Sigma^2)\dashrightarrow (X,B)$ be the crepant birational map of cluster type that maximizes $M(\psi)$.
Assume that $M(\psi)=s<k$. 
We may assume that the curves $C_1,\dots,C_s$
are non-exceptional over $\pp^2_\psi$
and the curves $C_{s+1},\dots,C_k$ are exceptional over $\pp^2_\psi$.
Let $p_1\colon (T,B_T)\rightarrow (\pp^2_\psi,\Sigma^2)$ be a toric projective birational morphism satisfying the following conditions:
\begin{enumerate}
    \item[(i)] every component of $B$ is a divisor on $T$, 
    \item[(ii)] every $\psi^{-1}$-exceptional curve can be obtained from $T$ by blowing up a sequence of points of $B_T$ that are not zero strata, and 
    \item[(iii)] If $t\in T$ is the center of a $\psi^{-1}$-exceptional curve, then $t$ is contained on a prime component $B_i$ of $B_T$ that is the section of a toric fibration 
    $f_t\colon T\rightarrow \pp^1$.
\end{enumerate}
By assumption, the curve $C_{s+1}$ is exceptional over $\pp^2_\psi$ and so it is exceptional over $T$.
Thus $C_{s+1}$ is a $\psi^{-1}$-exceptional curve.
Note that $C_{s+1}$ is not contained in $B$ so condition (ii) applies.
Henceforth, the curve $C_{s+1}$ can be obtained from $T$ by blowing up consecutively a point $t_{s+1}$ of $B_T$ that is not a zero strata.
Furthermore, by condition (iii), we know that there is a prime component $B_{s+1}$ of $B_T$ containing $t_{s+1}$ and a toric fibration $f_{s+1}\colon T\rightarrow \pp^1$ for which $B_{s+1}$ is a section.
Let $p\in \mathbb{G}_m\subset \pp^1$ be the image of $t_{s+1}$ in $\pp^1$ via $f_{s+1}$.
Let $F_{s+1}$ be the fiber of $f_t$ over $t_{s+1}$.
There are two cases: 
either $F_{s+1}$ agrees with one of the curves $C_i$ with $i\in \{1,\dots,s\}$ or it does not.
In the previous sentence, for simplicity, we are identifying the $C_i$'s with their strict transforms in $T$.
In what follows, we arise to a contradiction in either case. 

First, assume that $F_{s+1}=C_1$.
Let $p\colon X_0\rightarrow T$ be a projective birational map obtained by a sequence of blow ups of points of $B_T$ that are not strata such that $q\colon X_0\rightarrow X$ is a morphism.
We may assume that $p$ is minimal with the previous property. Let $x_1,\dots,x_r$ be the points blown up by $p$. We may assume $x_1=t_{s+1}$.
For each $i$ the curve $p^{-1}(x_i)=E_{i,1}\cup \dots E_{i,\ell_i}$ is a chain of rational curves.
By the minimality of $p$, only $E_{i,\ell_i}$ intersects $B_0$.
By the cluster type assumption of $\psi$, we know that ${\rm Ex}(p)$ is contained in $\cup_{i=1}^r p^{-1}(x_i) \cup B_0$.
As $B$ is ample, all the curves 
$E_{1,1},\dots,E_{1,s_1-1}$ must be contracted by $p$.
We conclude that $p(E_{1,s_1})$ and $p(F_{s+1})$ intersect along $X\setminus B$.
However, by construction, we have 
$p(E_{1,s_1})=C_{s+1}$ 
and $p(F_{s+1})=C_1$.
We conclude that $C_1$ and $C_{s+1}$ intersect along $X\setminus B$.
This contradicts Lemma~\ref{lem:config-curves-outside}.

Now, we assume that $F_{s+1}$ is not among the curves $C_1,\dots,C_k$.
We can perform a weighted blow up at $t_{s+1}$ that only extracts $C_{s+1}$.
Let $X_1\rightarrow X_0$ be such a blow up.
Let $(X_1,B_1)$ be the crepant pull-back of $(X_0,B_0)$. Note that we still have a fibration $X_1\rightarrow \pp^1$ and the strict transform of $F_{s+1}$ is degenerate over $\pp^1$.
Thus, there is a projective birational morphism
$X_1\rightarrow X_2$ that only contracts $F_{s+1}$.
Let $(X_2,B_2)$ be the log Calabi--Yau pair 
obtained by pushing-forward $B_1$ to $X_2$.
By construction, the log Calabi--Yau pair $(X_2,B_2)$ is toric
and the curves $C_1,\dots,C_{s+1}$ are divisors on $X_2$ not contained in $B_2$.
As $(X_2,B_2)$ is toric, we have a commutative diagram of crepant birational maps:
\[
\xymatrix{
(X_2,B_2)\ar@{-->}[r]^-{\pi_1} \ar@{-->}[d]_-{\pi_2} & (X,B) \\ 
(\pp^2_\xi, \Sigma^2)\ar@{-->}[ru]_-{\xi}  & \\ 
}
\]
where $\pi_1$ is a toric birational map
and $\xi$ is of cluster type.
As $\pi_1$ is a toric birational map, 
it is an isomorphism on $\mathbb{G}_m^2$.
In particular, the center of the curves $C_1,\dots,C_{s+1}$ are divisors on $\pp^2_\xi$.
We conclude that $M(\xi)=s+1>M(\psi)$.
This contradicts the minimality of $\psi$.

We conclude that $M(\phi)=k$ and so $\mathcal{C}(X,B)=2$. This finishes the proof.
\end{proof}

\section{Examples and questions}

In this section, we collect some examples and questions. 
We start with some examples of log Calabi--Yau pairs that exhibit some pathological behavior with respect to the torus exceptional degree.

\begin{example}\label{ex:not-cluster}
{\em 
In this example, we show a log Calabi--Yau surface $(X,B)$ of index one and birational complexity zero 
that is not of cluster type.
Let $(\pp^1\times \pp^1, Q+F)$
where $Q$ is a general curve of multi-degree $(2,1)$
and $F$ is a fiber of the projection onto the first component.
Then, the pair $(\pp^1\times \pp^1,Q+F)$ is a log Calabi--Yau pair.
Let $F'$ be a fiber that is tangent to $Q$ at a single point $x$.
Let $p$ be the image on $\pp^1$ of $x$.
We construct a new log Calabi--Yau pair following the next steps:
\begin{enumerate}
    \item we blow up $X$ at $x$,
    \item we blow up the intersection of the two components of the fiber over $p$, 
    \item we recursively blow up $n$ times the intersection of the strict transform of $Q$ and the fiber over $p$, and 
    \item we blow-down all the $(-2)$-curves contained in the fiber over $p$.
\end{enumerate}
By doing so, we obtain a new log Calabi--Yau pair $(X_n,B_n)$ of index one with a Mori fiber space to $\pp^1$.
Since $(\pp^1\times \pp^1,Q+F)\simeq_{\rm bir} (\pp^2,\Sigma^2)$, we have that $(X_n,B_n)$ has birational complexity zero.
Furthermore, the surface $X_n$ has a $D_n$ singularity.
By Proposition~\ref{prop:only-A-type}, we conclude that $(X_n,B_n)$ is not of cluster type.
}
\end{example}

\begin{example}\label{ex:total-ted-unbounded}
{\em  
In this example, we show that there is a sequence of log Calabi--Yau surfaces $(X_n,B_n)$ of index one and birational complexity zero for which 
the sequence $\mathcal{T}_t(X_n,B_n)$ diverges to infinite.

Let $(X,B)$ be a log Calabi--Yau pair of index one, birational complexity zero, that is not of cluster type.
For instance, we can consider the log Calabi--Yau surface constructed in Example~\ref{ex:not-cluster}.
Let $p$ be a point on $B$ that is not a zero-dimensional strata. 
Let $(X_n,B_n)$ be the log Calabi--Yau pair obtained by recursively blowing up $p$, and then blowing up $n$ times the intersection of the strict transform of $B$ with the previously introduced $(-1)$-curve.
We let $E_1\cup \dots E_{n+1}$ be the chain of rational curves that is the exceptional divisor 
of $X_n\rightarrow X$.
Assume by the sake of contradiction that $\mathcal{T}_t(X_n,B_n)\leq k$ for some constant $k$.
For each crepant birational map 
$\phi\colon (\pp^2,\Sigma^2)\dashrightarrow (X_n,B_n)$, we have $\mathcal{T}(\phi)\geq 1$.
Thus, we conclude that for each $n$ the log Calabi--Yau pair $(X_n,B_n)$ is divisorially covered by $\{\phi_{n,i}\}_{i\in I_n}$ with $|I_n|\leq k$.
For each $n$ and $i$, we let 
\[
|\phi_{n,i}|=
| \{ j\in \{1,\dots,n+1\} \mid 
\phi_{n,i}(\mathbb{G}_m^2 \setminus
{\rm Ex}(\phi_{n,i}) \cap E_j \neq \emptyset\}|.
\]
Observe that 
\[
\sum_{i\in I_n} |\phi_{n,i}| \geq n+1.
\]
Since $|I_n|\leq k$, for each $n$, there is one of the $\phi_{n,i}$'s for which $|\phi_{n,i}|\geq \lfloor \frac{n+1}{k}\rfloor$.
Let's assume that it is $\phi_{n,1}$.
Set $V_n:=\phi_{n,1}(\mathbb{G}_m^2 \setminus {\rm Ex}(\phi_{n,1}))$. 
Let $X_n\rightarrow X'_n$ be the projective birational contraction over $X$ that contracts all the curves $E_j$ that intersect
$V_n$ trivially.
Let $B'_n$ be the push-forward of $B_n$ on $X'_n$.
Thus, we get a commutative diagram
\[
\xymatrix{ 
 & (X_n,B_n) \ar[d]^-{\pi_n} \\ 
 & (X'_n, B'_n) \ar[d]^-{\pi'_n} \\
(\pp^2,\Sigma^2)\ar@{-->}[r] \ar@{-->}[ru]_-{\phi_{n,1}}
\ar@{-->}[ruu]^-{\phi_{n,1}}& (X,B). 
}
\]
Let $C_{n,1},\dots,C_{n,j_n}$ be closure on $X'_n$ of the prime components of 
\[
(X'_n\setminus B'_n) \setminus V_n.
\]
Since $\rho(X)=2$, we get $|j_n|\geq \lfloor \frac{n+1}{k} \rfloor-1$.
By Lemma~\ref{lem:affine-bounded-pi_1}, since $\mathcal{T}(\phi_{n,1})\leq k$, we conclude that 
\[
\dim H_1(V_n,\qq)\leq f(k)
\]
for certain constant $f(k)$ that only depends on $k$.
Since all the exceptional divisors of $\pi'_n$ are mapped to a smooth point in $X$, we get
\[
\dim H_1(V_n\setminus {\rm Ex}(\pi'_n),\qq) 
- \dim H_1(V_n,\qq) \leq 1. 
\]
In particular, we have 
\[
\dim H_1(V_n\setminus {\rm Ex}(\pi'_n),\qq)\leq f(k)+1.
\]
Let $U_n$ be the image on $X$ of $V_n\setminus {\rm Ex}(\pi'_n)$. Then, as observed above 
we have $\dim H_1(U_n,\qq)\leq f(k)+1$.
Note that
\[
U_n=X\setminus \left( B \cup_{j=1}^{j_n}
{\pi'_n}_*C_{n,j}\right).
\]
The sequence of divisors $\{{\pi'_n}_*C_{n,j}\}$ is infinite, as $j_n$ diverges to infinite. Thus, we get a contradiction of Lemma~\ref{lem:rank-inf} as $\dim H_1(U_n,\qq)$ is bounded above.
}
\end{example}

In Theorem~\ref{introthm:bcomp-zero-surf-aff-bounded}, we show that there exists a bounded family of affine surfaces $\mathcal{A}\rightarrow T$ such that for every log Calabi--Yau surface $(X,B)$ of index one and birational complexity zero
there is an embedding $\mathcal{A}_t\hookrightarrow X\setminus B$ for a suitable $t\in T$. The previous example shows that, in general, there is no upper bound $k$ such that every such $X\setminus B$ can be divisorially covered by at most $k$ open affines of the form $\mathcal{A}_t$. 
One may wonder if the log pair $(X,B)$ can be divisorially covered by an arbitrary number of open affines of the form $\mathcal{A}_t$. This motivates the following definition.

\begin{definition}
{\em  
Let $(X,B)$ be a log Calabi--Yau pair
of index one and birational complexity zero.
We define the {\em covering torus exceptional degree} of $(X,B)$, denoted by $\mathcal{T}_c(X,B)$, to be 
\[
\min \left\{ \max_{i \in I} \mathcal{T}(\phi_i) \mid  
\{\phi_i\}_{i\in I} \text{ 
is a set of crepant birational maps divisorially covering $(X,B)$}
\right\}.
\]
}
\end{definition}

From the definitions, it is clear that a log Calabi--Yau pair $(X,B)$ satisfies:
\[
\mathcal{T}_t(X,B) \geq
\mathcal{T}_c(X,B) \geq 
\mathcal{T}(X,B).
\]
While we expect $\mathcal{T}(X,B)$ to be bounded above in terms of the dimension (see Conjecture~\ref{conj:upper-bound-ted}), we know that $\mathcal{T}_t(X,B)$ can be arbitrarily large even for surfaces (see Example~\ref{ex:total-ted-unbounded}).
This leads to the following question.

\begin{question}
{\em 
Let $(X,B)$ be a log Calabi--Yau pair of dimension $n$ and birational complexity zero.
Is $\mathcal{T}_c(X,B)$ bounded above in terms of the dimension?
}   
\end{question}

Now, we turn to discuss some examples of bounded, affinely bounded, and birationally bounded families.

\begin{example}\label{ex:affine-bounded-not-bounded}
{\em
In this example, we show a family of varieties that is affinely bounded but is not bounded.
Let $n\geq 2$.
The set $\mathcal{T}_n$ of
$n$-dimensional projective toric varieties 
is affinely bounded.
Indeed, for every $n$-dimensional projective toric variety $X$ we have an open embedding
$\mathbb{G}_m^n \hookrightarrow X$.
On the other hand, the family $\mathcal{T}_n$ is not bounded as we can find an infinite sequence of smooth projective toric varieties $X_r$
with $\rho(X_r)=r$.
}
\end{example}

\begin{example}\label{ex:bir-bound-not-aff-bound}
{\em  
In this example, we show a set of varieties that is birationally bounded but is not 
affinely bounded. Let $\mathcal{R}_2$ be the set of all projective rational surfaces.
The set $\mathcal{R}_2$ is birationally bounded.
We show that is not affinely bounded.
Consider $\pi_1\colon \pp^1 \times \pp^1 \rightarrow \pp^1$ the projection to the first component.
Let $\{p_i\}_{i\in \nn}$ be an infinite sequence of points in $\pp^1$.
Let $X_k$ be the projective surface obtained as follows.
We blow up $\pp^1\times \pp^1$ at $q_i\in \pi_1^{-1}(p_i)$ for each $i\in \{1, \dots, k\}$ 
and then over each $p_i$ we blow up the intersection of the $(-1)$-curve and the strict transform of the fiber. 
Finally, we obtain $X_k$ by contracting all the $(-2)$-curves contained in the fibers of the fibration to $\pp^1$.
By doing so, we obtain a projective surface $X_k$ with $2k$ singularities of type $A_1$.
Furthermore, $X_k$ admits a Mori fiber space onto $\pp^1$ with precisely $2k$ fibers of multiplicity $2$.
By~\cite[Lemma 3.13]{FM23}, we conclude that there is a surjective homomorphism
\[
\pi_1^{\rm reg}(X_k) \rightarrow 
\pi_1\left( 
\pp^1, \frac{1}{2}(p_1+\dots+p_k) 
\right) \simeq 
\langle \gamma_1, \dots, \gamma_k 
\mid \gamma_1\cdots\gamma_k, \gamma_1^2,\dots, \gamma_k^2 \rangle. 
\]
Then the rank of $\pi_1^{\rm reg}(X_k)$ diverges as $k$ goes to infinite. 
Thus, the sequence of varieties $\{X_k\}_{k\in \nn}$ is birationally bounded
but it is not affinely bounded.
}
\end{example}

Finally, we explain an example of a log Calabi--Yau surface of large cluster type.

\begin{example}\label{ex:large-cluster-type}
{\em 
In this example, we show that for every integer $n\geq 0$ there exists a log Calabi--Yau pair
$(X,B)$ of dimension two with $\mathcal{C}(X,B)=n+1$.

Consider the log Calabi--Yau surface
$(X_n,B_n)$ obtained as follows. 
Start from $(\pp^1\times \pp^1,B)$ where
$B$ is the toric boundary.
Pick a point $p\in B$ which is not a zero stratum.
Let $B_1$ be the unique prime component of $B$
containing $p$.
Blow up $p$ and then recursively blow up $(n-1)$ times the intersection of the strict transform of $B$
with the previously extracted $(-1)$-curve.
By doing so, we get a log Calabi--Yau pair
$(X_n,B_n)$ that admits a fibration $f\colon X_n \rightarrow \pp^1$.
Furthermore, the fibers $f^{-1}(0)$ and $f^{-1}(\infty)$ are contained in the support of $B_n$.
All fibers of $f$ are isomorphic to $\pp^1$, 
except a single fiber, which is a chain of rational curves all of which are $(-2)$-curves except for $(-1)$-curves at either end.
We may assume that this unique reducible fiber is $f^{-1}(1)$ and write
$f^{-1}(1)=E_0 \cup E_1 \cup \dots \cup E_n$.
Let $\phi\colon (\pp^2,\Sigma^2)\dashrightarrow (X_n,B_n)$ be a crepant birational map of cluster type.
We apply Lemma~\ref{lem:diagram-cluster-type-surf}
and obtain a commutative diagram of projective birational maps
\[
\xymatrix{
(T_n,B_{T_n})\ar[d]^-{p_1} & (X_{n,0},B_{n,0})
\ar[l]_-{q}\ar[d]^-{p_2} \\
(\pp^2,\Sigma^2) & (X_n,B_n)\ar@{-->}[l]_-{\phi^{-1}}
}
\]
where the following conditions are satisfied:
\begin{itemize}
\item $p_1$ is a toric projective birational morphism,
\item $p_2$ is a projective birational morphism that only extracts log canonical centers of $(X,B)$, and 
\item $q$ is a projective birational morphism that only extracts canonical centers of $(T_n,B_{T_n})$.
\end{itemize}
By abuse of notation, we write $E_i$ for the strict transform of $E_i$ on $X_{n,0}$.
Note that $p_2$ only extracts log canonical centers of $(X_n,B_n)$. Let $f_0\colon X_{n,0}\rightarrow \pp^1$ be the induced fibration on $X_{n,0}$.
The fibers of $f_0$ are either rational curves intersecting $B_{n,0}$ twice, are contained in $B_{n,0}$, or are equal to $E_0\cup\dots\cup E_{n}$.
Furthermore, every curve in $X_{n,0}$ which is horizontal over $\pp^1$ is either contained in $B_{n,0}$ or intersects $B_{n,0}$ at least twice.
Indeed, every curve on $X_{n,0}$ which is horizontal over $\pp^1$ must intersect $f_{n,0}^{-1}(0)$ and $f_{n,0}^{-1}(\infty)$ and both fibers are contained in $B_{n,0}$.
By Lemma~\ref{lem:contraction-S^1}, we conclude that all the exceptional divisors of $q$ are vertical over $\pp^1$.
Since every fiber $f_0^{-1}(t)$, with $t\in \mathbb{G}_m$ and $t\neq 1$, intersects $B_{n,0}$ twice, we conclude that every exceptional divisor of $q$ must be contained in $E_0\cup \dots \cup E_{n}$.
Since $(T,B_T)$ is toric, by a complexity computation, we conclude that $q$ must contract $n$ of the curves $E_0,\dots, E_{n}$.
Thus, for every crepant birational map $\phi \colon (\pp^2,\Sigma^2) \dashrightarrow (X_n,B_n)$
of cluster type, the image $\phi(\mathbb{G}_m^2)$ can intersect at most one of the curves $E_0\cup\dots \cup E_{n}$.
Hence, we conclude that at least $n+1$ almost tori are required to cover $X_n\setminus B_n$ divisorially.
This means that $\mathcal{C}(X_n,B_n)\geq n+1$.
On the other hand, by contracting precisely $n$ of the curves $E_0,\dots,E_n$, we observe that $n+1$ almost tori cover $X_n\setminus B_n$ divisorially.
Thus, we get $\mathcal{C}(X_n,B_n)=n+1$.
}
\end{example}

The previous example shows that the cluster type of
a log Calabi--Yau surface can be an arbitrary positive integer.
In contrast, 
Theorem~\ref{introthm:cluster-Fsurf} implies that $\mathcal{C}(X,B)\leq 2$ for a cluster type log Calabi--Yau pair $(X,B)$ whenever $X$ is a Fano surface.
It is natural to expect similar behavior in higher dimensions.
This motivates the following question.

\begin{question}
{\em 
Let $X$ be a $n$-dimensional Fano variety.
Let $(X,B)$ be a log Calabi--Yau pair of cluster type.
Can we bound $\mathcal{C}(X,B)$ above in terms of $n$?
}
\end{question}

\bibliographystyle{habbvr}
\bibliography{references}

\end{document}